\documentclass[a4paper,12pt]{article}
\usepackage{amssymb,amsmath,amsthm}
\usepackage[left=.75 in, right=.75 in,top=1.5 in, bottom=1.5 in]{geometry} 
\usepackage{graphicx}
\usepackage{tikz}
\usepackage[labelfont=bf,textfont=md]{caption}
\usepackage{float}
\usepackage{enumitem}
\usepackage{multicol}
\usepackage{subcaption}
\usepackage{changepage}
\usepackage{hyperref}
\usetikzlibrary{decorations.pathreplacing}
\usepackage{titling}
\usepackage{mathtools}

\def\H1{{_0}H^1(\Omega)}

\def\k{\mathcal{K}}

\def\N{\mathbb{N}}

\def\Q{\mathbb{Q}}
\def\R{\mathbb{R}}

\def\Xint#1{\mathchoice
{\XXint\displaystyle\textstyle{#1}}%
{\XXint\textstyle\scriptstyle{#1}}%
{\XXint\scriptstyle\scriptscriptstyle{#1}}%
{\XXint\scriptscriptstyle\scriptscriptstyle{#1}}%
\!\int}
\def\XXint#1#2#3{{\setbox0=\hbox{$#1{#2#3}{\int}$ }
\vcenter{\hbox{$#2#3$ }}\kern-.6\wd0}}
\def\intbar{\Xint-}

\DeclareMathOperator*{\essinf}{ess\,inf}

\newtheorem{lem}{Lemma}[section]

\newtheorem{prop}[lem]{Proposition}
\newtheorem{thm}[lem]{Theorem}
\newtheorem{remark}[lem]{Remark}

\newtheorem*{prop*}{Proposition}
\newtheorem*{thm*}{Theorem}
\newtheorem*{def*}{Definition}
\newtheorem*{lem*}{Lemma}

\newcommand{\sm}{\setminus}

\newcommand{\ov}{\overline}

\newcommand{\wkto}{\rightharpoonup}

\numberwithin{equation}{section}

\title{On $\Gamma-$Convergence of a Variational Model for Lithium-Ion Batteries}
\author{Kerrek Stinson \\  kstinson@andrew.cmu.edu \\ Carnegie Mellon University}

\begin{document}
\pagestyle{plain}
\setcounter{page}{1}
 
\begin{titlingpage}
    \maketitle
    \begin{abstract}
        A singularly perturbed phase field model used to model lithium-ion batteries including chemical and elastic effects is considered. The underlying energy is given by
        $$I_\epsilon [u,c ] := \int_{\Omega}\left(\frac{1}{\epsilon}f(c)+\epsilon\|\nabla c\|^2+\frac{1}{\epsilon}\mathbb{C}(e(u)-ce_0):(e(u)-ce_0)\right) dx, $$
        where $f$ is a double well potential, $\mathbb{C}$ is a symmetric positive definite fourth order tensor, $c$ is the normalized lithium-ion density, and $u$ is the material displacement. The integrand contains elements close to those in energy functionals arising in both the theory of fluid-fluid and solid-solid phase transitions. For a strictly star-shaped, Lipschitz domain $\Omega \subset \R^2,$ it is proven that $\Gamma - \lim_{\epsilon\to 0} I_\epsilon = I_0,$ where $I_0$ is finite only for pairs $(u,c)$ such that $f(c) = 0$ and the symmetrized gradient $e(u) = ce_0$ almost everywhere. Furthermore, $I_0$ is characterized as the integral of an anisotropic interfacial energy density over sharp interfaces given by the jumpset of $c.$
    \end{abstract}
 \noindent    \textbf{Key words:} Gamma convergence, lithium-ion batteries, linear elasticity
    
  \noindent   \textbf{AMS Classifications:} 74G65, 49J45, 74N99
\end{titlingpage}

\newpage

\section{Introduction}
The lithium-ion battery is a fundamental tool in modern technology and the intertwined challenge of harnessing renewable energy, with applications extending from mobile phones to hybrid cars. In recognition of this importance, the 2019 Nobel Prize in Chemistry was awarded to Goodenough, Whittingham, and Yoshino for their pioneering works in the development of lithium-ion batteries \cite{nobelPrize}. Motivated by the eminence of lithium-ion batteries, we study a mathematical model that underlies their capacity. A prominent performance limitation of lithium-ion batteries is their short life-cycle resulting from the electrochemical processes governing the battery which induce phase transitions. Elaborating on this, during the process of charging, lithium-ions intercalate into the host structure of the cathode. This intercalation is not homogeneous and undergoes phase separation, that is, lithium-ions form areas of high concentration and low concentration with sharp phase transitions between these regions. These phase transitions induce a strain on the host material which, ultimately, leads to its degradation. Damage of the cathode's host material leads to a decrease in battery performance and limited life-cycle (see \cite{Bazant-Theory2013}, \cite{Dal2015-Comp}, and references therein).

Understanding the onset of phase transitions is, therefore, imperative to improving battery performance, and much work has been done in this direction. Contemporary paradigms for modeling lithium-ion batteries are moving towards the incorporation of phase field models, also known as diffuse interface models (see, e.g., \cite{singh2008intercalation}, \cite{cogswell2012coherency}, \cite{bai2011suppression}, \cite{Balluffi2005KineticsOM}, \cite{Nauman2001NonlinearDA}). These phase field models are governed by global energy functionals, which have regular inputs (e.g. Sobolev functions). As noted in \cite{Bazant-Theory2013}, the phase field field model is robust, allowing for electrochemically consistent models for the time evolution of lithium-ion batteries. Competing models include the shrinking core model and the sharp interface model; however, as noted in Burch et. al. \cite{Burch-PhaseTransformation2008}, the shrinking core model fails to capture fundamental qualitative behavior. Furthermore, in \cite{HanVanDerVen2004} it is proposed that the phase field model may provide a more accurate numerical analysis of the problem than the sharp interface model, which seeks to model the evolution of the phase boundary as a free boundary problem (see \cite{caginalp1991}; see also \cite{Acharya-msThesis}, and references therein, for benefits of the phase field model).

In this paper we study a variational model introduced by Cogswell and Bazant in \cite{cogswell2012coherency} (see also \cite{Bazant-Theory2013}, \cite{Bazant-PhaseSepDyn2014}, \cite{singh2008intercalation}, \cite{burch2009size}). For a fixed domain $\Omega\subset \R^2$, we consider a phase field model for which the free energy functional is given by \begin{align*}
I &[u,c,\Omega] :=  \int_{\Omega}\left( \bar f(c)+\rho\|\nabla c\|^2+\mathbb{C}(e(u)-ce_0):(e(u)-ce_0) \right) dz
\end{align*} with 
\begin{equation} \label{chemPotFunc}
\bar f(s)  := \omega s(1-s)+KT(s\log(s) +(1-s)\log(1-s)),  \ \ \ \ \ s\in [0,1].
\end{equation} 
Here $c:\Omega \to [0,1]$ stands for the normalized density of lithium-ions, and $u:\Omega \to \R^2$ represents the material displacement with symmetrized gradient $e(u):= \frac{\nabla u +\nabla u^T}{2}$, $\omega\in \R$ is a regular solution parameter (enthalpy of mixing), $e_0\in \R^{2\times 2}$ is the lattice misfit, $K>0$ is the Boltzman constant, $T>0$ is the absolute temperature, $\rho>0$ is a constant associated with interfacial energy scaling with interface width (see \cite{BallJames-FinePhase1987}, \cite{kohn1994-surface}, \cite{Muller1999-VariationalLecture}, \cite{bhattacharya-martensiteBook}, and references therein), and $\mathbb{C}$ is a symmetric, positive definite, fourth order tensor, that captures the material constants (stiffness). Note the tensor $\mathbb{C}$ is defined to be positive definite as follows
\begin{align} \label{coercivity}
\mathbb{C}:\R^{2\times 2}\to \R^{2\times 2}_{\rm{sym}}, \quad   \mathbb{C}( \xi):\xi >0 \text{ for all } \xi \in \R^{2\times 2}_{\rm{sym}} \text{ with } \xi \neq 0.
\end{align}

Adding a constant and letting $\rho  := \epsilon^2$, we rescale the functional by $1/\epsilon$ to consider the collection of functionals $\{I_\epsilon\}_{\epsilon>0}$ on $H^1(\Omega,\R^2)\times L^2(\Omega,[0,1])$ defined as  
\begin{equation}\label{energyFunc}
\begin{aligned}
 & I_{ \epsilon} [u,c,\Omega] := \\
& \begin{cases} 
 \int_{\Omega}\left (\frac{1}{\epsilon}f(c)+\epsilon\|\nabla c\|^2+\frac{1}{\epsilon}\mathbb{C}(e(u)-ce_0):(e(u)-ce_0) \right) dz & (u,c)\in H^1(\Omega,\R^2)\times H^1(\Omega,[0,1]), \\
 \infty & \text{otherwise,}
 \end{cases}
 \end{aligned}
\end{equation} where 
\begin{equation}\label{wellFunc}
f(s) := \bar f(s)-\min_{t\in [0,1]} \bar f(t), \ \ \ \ \ s\in [0,1]
\end{equation} 
is a well function. We wish to consider the asymptotic behavior of this collection of energies as $\epsilon\to 0$ (i.e., when the interfacial width goes to $0$). This analysis will, in some capacity, mathematically validate the numerical solutions witnessing phase separation for small interfacial widths as seen by Bazant and Cogswell in \cite{cogswell2012coherency}. 

To study the asymptotic behavior, we will use the notion of $\Gamma-$convergence, as introduced by De Giorgi in \cite{DeGiorgi-GammaConv}. $\Gamma-$convergence was first used by Modica and Mortola in \cite{ModicaMortola} to study the class of functionals arising in the Cahn-Hilliard theory of fluid-fluid transitions given by $$E_\epsilon[c,\Omega] := \int_\Omega \left( \frac{1}{\epsilon}W(c) + \epsilon\|\nabla c\|^2 \right) dz, \ \ \ c\in H^1(\Omega,\R),$$ where $W$ is a double well function and $\Omega\subset \R^N$ (see also the foundational work by Cahn and Hilliard \cite{CahnHilliard}). Herein, they showed that $\Gamma - \lim_{\epsilon\to 0} E_\epsilon = E_0,$ where $E_0(c) := C\text{Per}_\Omega(c),$ with $\text{Per}_\Omega(c)$, the perimeter in $\Omega$ of one of the phases of $c$, taken to be $\infty$ if $c$ is not of finite perimeter. See also \cite{FonsecaTartar1989}, \cite{barroso1994-Anisotropic}, \cite{Ambrosio1990-metric}, and references therein.

More recently, a variety of work has been directed at analyzing classes of functionals given by 
\begin{equation}\label{csEnergyFunc}
F_\epsilon[u,\Omega] := \int_\Omega \left( \frac{1}{\epsilon}W(\nabla u) + \epsilon\|\nabla^2 u\|^2 \right)  dz, \ \ \ u\in H^2(\Omega,\R^N),
\end{equation} with $\Omega\subset \R^N,$ which arise in the theory of solid-solid phase transitions \cite{bhattacharya-martensiteBook}. Accounting for frame indifference in a geometrically nonlinear framework, it is necessary to consider $W$ satisfying the well condition $W(G) = 0$ if and only if $G\in \text{SO}(N)A\cup\text{SO}(N)B$ for matrices $A,B\in \R^{N\times  N},$ where $SO(N)$ is the special orthogonal group. To guarantee existence of nonaffine functions for which the limiting energy is finite, the wells must satisfy Hadamard's rank-one compatibility condition given by $QA-B = a\otimes \nu$ for some $Q\in \text{SO}(N),$ and $a,\nu \in \R^N$ (see \cite{BallJames-FinePhase1987}, \cite{Dolzmann1995-microstructure}). As an initial step in \cite{ContiFonsecaLeoni-gammConv2grad}, Conti et. al. treat the case of a double well function $W$ disregarding frame indifference, meaning $W(G)=0$ if and only if $G=A$ or $G=B$, concluding that $\{F_{\epsilon}\}_{\epsilon>0}$ $\Gamma-$converges to a functional reminiscent of $F_0$ defined in (\ref{def:f0func}). Convergence of a case intermediate to $E_\epsilon$ and $F_\epsilon$ is considered by Fonseca and Mantegazza \cite{FonsecaMantegazza-SecondOrder2000} wherein the nonconvex integrand of $F_\epsilon$ is replaced by $\frac{1}{\epsilon}W(u).$ Many promising results regarding convergence of $F_\epsilon$ when it is the Eikonal functional, that is $W(G) := (1-\|G\|^2)^2$, have been obtained, although the $\Gamma-$limit is still yet to be identified (see \cite{DeLellis2003-structure}, \cite{desimone_muller_kohn_otto_2001}). 

Restricted to a strictly star-shaped Lipschitz domain $\Omega\subset \R^2$, Conti and Schweizer in \cite{ContiSchweizer-Linear} address the problem of frame indifference in a geometrically linear framework, that is when $W$ is invariant under the tangent space of $\text{SO}(2)$ or, equivalently, satisfies the well condition $W(G) = 0$ if and only if $\frac{G+G^T}{2} \in \{A\}\cup \{B\}$. Conti and Schweizer conclude that the functionals $\{F_\epsilon\}_{\epsilon>0}$ $\Gamma-$converge to 
\begin{equation}\label{def:f0func}
F_0[u,\Omega] := \begin{cases}
\int_{J_{e(u)}} k(\nu)\ d\mathcal{H}^1 & \text{ if } e(u)\in BV(\Omega,\{A,B\}), \\
\infty & \text{ otherwise,}
\end{cases}
\end{equation} 
where $J_{e(u)}$ is the associated jumpset with normal $\nu$, and $k(\nu)$ is the effective anisotropic interfacial energy density. Again, the existence of displacement with non-constant symmetrized gradient exactly on the two wells requires a rank-one connectivity property. To be precise, there is some skew-symmetric matrix $S$ such that $A-B+S$ is rank one (see Proposition \ref{prop:detcond}). Furthermore, the condition that $e(u)\in BV(\Omega,\{A,B\})$ forces considerable restriction on the functions for which $F_0[u]<\infty.$ Specifically, each interface of $J_{e(u)}$ has a single normal (out of two choices) and extends to the boundary of $\Omega$. Consequently $u$ behaves like a laminate (see Theorem \ref{struct}).

Furthermore in \cite{conti-SO2rigid}, with $N=2,$ Conti and Schweizer analyze the case of a geometrically nonlinear framework with a result analogous to the linear case. In order to extend this result to higher dimensions, in \cite{davoli2018twowell} Davoli and Friedrich analyze the energy 
$$\int_\Omega \left( \frac{1}{\epsilon}W(\nabla u) + \epsilon\|\nabla^2 u\|^2  + \eta(\epsilon)(\|\nabla^2 u\|^2 - |\partial_{N}^2 u|^2) \right)  dz, \ \ \ u\in H^2(\Omega,\R^N) $$
utilizing sophisticated rigidity results for incompatible vector fields (see \cite{Muller-IncompatGeoRigidity}, \cite{Chambolle-piecewiseRigid}, \cite{lauteri2017geometric}). Here, it is assumed that the two wells of $W$ are $0$ and $SO(N)e_N\otimes e_N$. Furthermore, the last term in the energy specifically penalizes change in the displacement orthogonal to $e_N$, and it follows that there is a single relevant interfacial normal, $e_N$. This is in contrast to the two possible interfacial normals that arise in the analysis of $\{F_\epsilon\}_{\epsilon>0}$ (see Theorem \ref{struct}). Here $\eta (\epsilon)\to \infty$ as $\epsilon\to 0$, leaving the identification of the $\Gamma-$limit of $\{F_\epsilon\}_{\epsilon>0}$ in arbitrary dimensions an open problem. 

Looking towards applications to fracture mechanics, Bellettini et. al. \cite{bellettini2013-gammaLim} analyze $\Gamma-$convergence of the energy functionals
$$
\int_\Omega \left( \frac{1}{\epsilon \phi(1/\epsilon)}\phi(\|\nabla u\|) + \epsilon^3 \|\nabla^2 u\|^2 \right) dz  , \ \ \ u\in H^2(\Omega,\R^N)
 $$ where $\phi :[0,\infty)\to [0,\infty)$ is continuous, nondecreasing, has sublinear growth at infinity, and satisfies $\phi^{-1}(\{0\}) = \{0\}$. As noted by the authors, this energy may viewed as a special case of (\ref{csEnergyFunc}) where the wells of $W$ are at $0$ and $\infty.$ 

The integrand in the energy $I_\epsilon$ bears clear similarities to the integrands of both functionals $E_\epsilon$ and $F_\epsilon$. In our analysis of the $\Gamma-$convergence of the functionals $I_\epsilon,$ we will use many of the ideas put forth in the $\Gamma-$convergence analyses of both $E_\epsilon$ by Modica and Mortola in \cite{ModicaMortola} and $F_\epsilon$ by Conti and Schweizer in \cite{ContiSchweizer-Linear}. 

We now introduce some terminology allowing us to state the main results of this paper. Let $\mu_0 \in (0,1)$ and $\mu_1 = 1-\mu_0 \in (0,1)$ be the two wells of $f$ (see Proposition \ref{prop:ffunc}).  
In view of Remark \ref{rmk:symAssump} and Proposition \ref{prop:detcond}, we assume that 
\begin{equation} \label{detcond} \text{det}(e_0)\leq 0, \quad \quad  e_0\in \R^{2\times 2}_{\rm{sym}},
\end{equation}
and consequently there are one or two choices (up to sign) of $\nu \in S^1$ such that 
\begin{equation} \label{R1connection}
  S_\nu := a \otimes \nu - (\mu_1-\mu_0)e_0
  \end{equation} is skew symmetric for some $a\in \R^2$ (see Section \ref{secmathprelim}).
Letting $\Q_\nu$ be a unit square in $\R^2$ centered at the origin with two sides parallel to $\nu$, we define the following interfacial energy density
\begin{equation}\label{interfaceenergyconst} 
\begin{aligned}
\k(\nu) := \inf\{\liminf\limits_{i\to \infty}I_{\epsilon_i}[u_i,c_i,\Q_\nu]:\epsilon_i\to 0,  u_i\in &  H^1(\Q_\nu,\R^2),  u_i\to \bar u_\nu \text{ in }H^1(\Q_\nu,\R^2),   \\
& c_i\in H^1(\Q_\nu,[0,1]), c_i\to \bar c_\nu \text{ in } L^2(\Q_\nu) \}, 
\end{aligned}
\end{equation} with
\begin{equation} \label{unucnu}
\begin{aligned}
&\bar u_\nu(x,y)  := \begin{cases} 
      \mu_0 e_0(x,y)^T &  \text{ if }(x,y)\cdot \nu<0,\\
      (\mu_1e_0+S_\nu)(x,y)^T & \text{ if } (x,y)\cdot \nu>0,
   \end{cases}\\
   &\bar c_\nu(x,y)  := \begin{cases} 
      \mu_0 & \text{ if } (x,y)\cdot \nu<0,\\
      \mu_1 & \text{ if } (x,y)\cdot \nu>0.
   \end{cases}
\end{aligned} 
\end{equation}
Note that $\bar u_\nu$ is Lipschitz by virtue of (\ref{R1connection}). With these definitions in hand, we now state the main results of this paper:

\begin{thm} \label{maintheorem}
Let $\Omega\subset \R^2$ be an open, bounded, star-shaped domain with Lipschitz continuous boundary, and assume that (\ref{coercivity}) and (\ref{detcond}) hold. Considering the strong topology of $H^1(\Omega,\R^2)\times L^2(\Omega,[0,1]),$ we have $$\Gamma - \lim_{\epsilon\to 0} I_{\epsilon} = I_0,$$ where $I_\epsilon$ is defined in (\ref{energyFunc}), and 
\begin{equation}\label{energyFuncGLim}
\begin{aligned}
 I_0[u,c,\Omega] :=  \begin{cases} 
      \int_{J_{c}} \k(\nu) \ d\mathcal{H}^1 &  c\in BV(\Omega;\{\mu_0,\mu_1\}),\ u\in H^1(\Omega;\R^2), \ \ e(u) = ce_0,\\
    \infty & \text{otherwise},
   \end{cases}
  \end{aligned}
  \end{equation}
where $J_c$ is the jumpset for $c$ with normal $\nu$, and $\mu_0$ and $\mu_1$ are the wells of $f$ (see (\ref{wellFunc})).
\end{thm}

We note that in the above theorem we have restricted the functions $c$ to map into $[0,1],$ a physically meaningful constraint as $c$ is the normalized lithium-ion density. 

Furthermore, it is natural to consider specific mass constraints on the imposed on the lithium-ions. Explicitly, let $\{m_\epsilon \}_{\epsilon>0}\subset [0,1]$ be a net converging to $m_0\in [\mu_0,\mu_1]$ as $\epsilon\to 0$, and consider $\Gamma$-convergence restricting $c_\epsilon$ to satisfy $\intbar_\Omega c_\epsilon \ dx \ dy = m_\epsilon.$ We then have:

\begin{thm} \label{maintheoremconstraint}
The results of Theorem \ref{maintheorem} still hold under the restriction that $\Gamma - $convergence is performed restricting $c_\epsilon$ to satisfy $$\intbar_\Omega c_\epsilon \ dz = m_\epsilon,$$ where $m_\epsilon\in [0,1].$ 
\end{thm}

We comment that this result specifically depends on the split structure wherein $\Gamma-$convergence relies on both the convergence of $u_\epsilon$ and $c_\epsilon.$ The analogous constraint in the case of energies such as $F_\epsilon$ would be a mass constraint imposed on the gradient, but such gradient restrictions impose more difficulties in the explicit construction of low energy sequences.

In Section \ref{secmathprelim} we introduce basic definitions and present some results about the functional $I_{\epsilon}$. With these in hand, in Section \ref{seccompact} we consider the compactness of the energy functionals, i.e., if $I_{\epsilon_i}[u_i,c_i,\Omega]\leq C <\infty$ for all $i\in \N,$ for which topologies do $\{u_i\}$ and $\{c_i\}$ converge? We conclude that, up to subsequences, $\{u_i\}$ and $\{c_i\}$ strongly converge in $H^1$ and $L^2$, respectively. This naturally motivates us to consider $\Gamma-$convergence for the energy functionals with strong convergence of $(u_i,c_i)$ in $H^1(\Omega,\R^2)\times L^2(\Omega,[0,1]).$ In Section \ref{secliminf} we prove the associated limit inferior bound showing that for any sequence $\epsilon_i\to 0$, for all $(u_i,c_i)\to (u,c)$ in $H^1\times L^2,$ we have $$\liminf_{i \to \infty} I_{\epsilon_i}[u_i,c_i,\Omega]\geq I_{0}[u,c,\Omega].$$ To conclude Theorem \ref{maintheorem}, it remains to prove that there is a recovery sequence for any pair $(u,c)\in H^1(\Omega,\R^2)\times L^2(\Omega,[0,1])$ such that $I_0[u,c,\Omega]<\infty$. To do this, we will need a precise characterization of the interfacial energy in terms of sequences which are affine away from the interface. We prove this characterization in Section \ref{sec:charInterface}. In Section \ref{seclimsup} we critically utilize this characterization to prove that for any $(u,c)\in H^1(\Omega,\R^2)\times L^2(\Omega,[0,1])$ there are $(u_i,c_i)\in H^1(\Omega,\R^2)\times L^2(\Omega,[0,1])$ strongly converging to $(u,c)$ with $$\lim_{i\to \infty} I_{\epsilon_i}[u_i,c_i,\Omega] = I_{0}[u,c,\Omega].$$ Lastly, in Section \ref{secmassconstraint} we extend Theorem \ref{maintheorem} to the case of mass constraints (see Theorem \ref{maintheoremconstraint}).

The primary contribution of this paper to the existing literature on phase field models for lithium-ion batteries is the mathematical validation of the numerical solutions witnessing phase separation for small interfacial widths as seen by Bazant and Cogswell \cite{cogswell2012coherency}. The primary mathematical contribution of this paper is in connecting analysis of the functional $I_\epsilon$ to the treatment of the functional $F_{\epsilon}.$ Apriori, the latter connection is not clear as no second order terms appear in $I_\epsilon$ and $I_\epsilon[u,c,\Omega]$ possesses the integrand term $$\|e(u)-ce_0\|^2$$ which is not a well function. However this term is similar to the well function $W(\nabla u) := \min\{\|e(u)-\mu_0e_0\|^2,\|e(u)-\mu_1e_0\|^2\}$, and this similarity is exploited to crucially apply the rigidity analysis of Conti and Schweizer in \cite{ContiSchweizer-Linear}.

\section{Preliminaries}\label{secmathprelim}

We first introduce some notation that will be used throughout the paper. We write $z = (x,y)\in \R^2,$ and we denote by $e_x$ and $e_y$ the standard basis vectors in $\R^2.$ For a set $D\subset \R^2$, we define $\chi_D:\R^2\to \{0,1\}$ to be the indicator function of $D.$ We denote the convex hull of a set $D\subset \R^2$ by $\text{conv}(D).$ Given $\phi\in \R,$ we further define the skew symmetric matrix 
\begin{equation}\label{def:skewsymRot}
R_\phi : = \begin{bmatrix}
   0      & -\phi \\
    \phi      & 0 \end{bmatrix}.
    \end{equation} For $u\in H^1(\Omega,\R^2),$ we define the symmetrized gradient $e(u):=\frac{\nabla u+(\nabla u)^T}{2}.$ For a function $c\in BV(\Omega,\R),$ we let $J_c$ denote the jumpset of $c$ (see \cite{AmbrosioFuscoPallara},\cite{EvansGariepy}). 
We will occasionally drop reference to the domain or range in a function norm, e.g., $\|u\|_{H^1(\Omega,\R^2)} = \|u\|_{H^1(\Omega)} = \|u\|_{H^1}.$ If a norm is written without a function space subscript, it refers to the euclidean norm of the vector or matrix.

We note throughout the following that we will consider the class of functionals $\{I_\epsilon\}_{\epsilon>0}$ (defined by (\ref{energyFunc})) as defined on $H^1(\Omega,\R^2)\times L^2(\Omega,[0,1])\times \mathcal{A}(\R^2),$ where $\mathcal{A}(\R^2)$ is the collection of all open subsets of $\R^2.$

We will make use of the exact structure of the well function $f$ (see (\ref{chemPotFunc}) and (\ref{wellFunc})).
\begin{prop}\label{prop:ffunc}
Let $f$ be defined as in (\ref{chemPotFunc}). The following holds:
\begin{enumerate}[label=\roman*)]
\item If $\omega\leq 2KT,$ then $f$ is a single-well function.
\item If $\omega>2KT$, then $f$ is a double-well function with super-quadratic wells at $\mu_0 \in (0,1/2)$ and $\mu_1 = 1-\mu_0 \in (1/2,1).$
\end{enumerate}
\end{prop}
\begin{proof}
 By definition of absolute temperature and the Boltzmann constant, we note that it always holds that $KT\geq 0$. However, there are no restrictions on the sign of $\omega$. In the case $ \omega\leq 0$, we note that $f$ is decreasing on the interval $[0,1/2]$ and increasing on the interval $[1/2,1],$ as observed by a direct inspection of the derivative $$\frac{d}{ds}f(s) =\omega(1-2s)+KT\log\Big(\frac{s}{1-s} \Big).$$  Consequently $f$ is a single-well function.

For the case of $\omega >0$, we note that 
\begin{equation}\label{2ndderiv}
\frac{d^2}{ds^2}f(s) =-2\omega+\frac{KT}{s(1-s)},
\end{equation} which has at most $2$ zeros. Hence, $f$ necessarily has zero, one, or two inflection points.

In the case of zero inflection points, that is when $\omega<2KT$, $f$ has a single well (minimum) at $1/2$, as the derivative blows up to negative infinity at the $0$ boundary point.

In the case of one inflection point, that is when $\omega=2KT$, symmetry implies it occurs at $1/2,$ and this is the minimizer. We note the well is not super-quadratic.

In the case of two inflection points, that is when $\omega>2KT$, we must have that $f$ is a double well function with superquadratic wells. Note that the inflections must occur on the interior by equation (\ref{2ndderiv}) and there must exist $\mu_0$ such that a minimum is obtained. If this minimum is obtained at $\mu_0=1/2,$ we cannot have two inflection points. To see this, note we cannot have any local min/maxes away from the global minimizer or else we contradict the number of inflection points. Thus, at the inflection point, $\frac{d}{ds} f\leq 0$. As $\frac{d^2}{ds^2}f$ is the reciprocal of a quadratic plus a constant, it changes signs at the inflection point, and consequently $\frac{d^2}{ds^2}f<0$ after this point. But this implies $\frac{d}{ds}f(1/2)<0,$ a contradiction. Consequently, the minimum is obtained for some $\mu_0 \neq 1/2$ and $\mu_1 = 1-\mu_0.$ As there are at least two inflection points between every minimizer, these are the only minimizers (local or global). We further note that the function $f$ cannot inflect at $\mu_0,$ else the derivative is only positive between $\mu_0$ and $\mu_1.$ From this, it follows that we may write $[0,1]$ as the union of $I_1:= [0,\mu_0], \  I_2:=[\mu_0,1/2], \ I_3:=[1/2,\mu_1],$ and $I_4:=[\mu_1,1],$ where $f$ is decreasing on $I_1$ and $I_3$ and increasing on $I_2$ and $I_4.$ As in the case of zero inflection points, we have that $\frac{d^2}{ds^2}f(\mu_0)>0$, and we may apply the fundamental theorem of calculus to find a desired quadratic function to show that $f$ is super-quadratic at the wells.

\end{proof}

In the case in which $f$ is a single-well function, phase separation will not be witnessed (see \cite{Bazant-PhaseSepDyn2014}). The analysis of this case is simple as the functions for which $I_0$ is finite still belong to Sobolev spaces, and we do not focus on it. Consequently, in what follows we assume $f$ is a double well, with wells $\mu_0$ and $\mu_1$ satisfying 
\begin{equation}\label{wells}
0<\mu_0<1/2<\mu_1<1,
\end{equation}  and 
\begin{equation}\label{fcoeffCond}
\omega>2KT.
\end{equation}

Before invoking (\ref{detcond}) to simplify the functional $I_\epsilon$, we provide a justification of this assumption (see also \cite{BallJames-FinePhase1987}, \cite{Dolzmann1995-microstructure}). 
\begin{remark}\label{rmk:symAssump}
We note that by property (\ref{coercivity}), $\mathbb{C}(\R^{2\times 2}_{\rm{skew}}) = \{0\}.$ Furthermore we recall that symmetric and skew-symmetric matrices are orthogonal with respect to the Frobenius inner product. Uniquely decomposing the lattice misfit matrix as $e_0 = e^{\rm{sym}}_0 +e^{\rm{skew}}_0,$ with $e^{\rm{sym}}_0 \in \R^{2\times 2}_{\rm{sym}}$ and $e^{\rm{skew}}_0\in \R^{2\times 2}_{\rm{skew}}$, it follows 
$$\mathbb{C}(e(u)-ce_0):(e(u)-ce_0) = \mathbb{C}(e(u)-ce_0^{\rm{sym}}):(e(u)-ce_0^{\rm{sym}}). $$ Consequently, the assumption $e_0\in \R^{2\times 2}_{\rm{sym}}$ in (\ref{detcond}) occurs without loss of generality.
\end{remark}
\begin{prop} \label{prop:detcond} Suppose there is non-affine $u\in C(\Omega,\R^2)$ which is piecewise $C^1$ with the jumpset of $\nabla u$ given by a disjoint union of $C^1$ manifolds, and $e(u)\in \{\mu_0,\mu_1\}e_0$ where $\mu_0,\mu_1$ satisfy (\ref{wells}) and $e_0\in \R^{2\times 2}_{\rm{sym}}$. Then (\ref{detcond}) holds.
\end{prop}
\begin{proof}
We may consider the tangent derivative of $u$ at a point $z_0$ on interface separating regions where $e(u) = \mu_0e_0$ and $e(u) =  \mu_1 e_0$. Computing the tangent derivative in the direction $t\in \R^2$ from both sides of the interface, we find $$ (\mu_0 e_0 +S)t = \nabla u(z_0) t = (\mu_1e_0 + S')t$$ for some skew-symmetric matrices $S$ and $S'$. 
Rearranging, we have $$((\mu_1-\mu_0)e_0 +S_\nu)t = 0 $$ with $S_\nu = \begin{bmatrix}
   0      & s \\
    -s      & 0 \end{bmatrix} :=  S'-S.$ 
It follows that \begin{equation} \label{R1connection2}
(\mu_1-\mu_0)e_0 +S_\nu = a \otimes \nu
\end{equation} for some vector $a\in \R^2$ and $\nu\in S^1$ normal to the interface (i.e., normal to $t$). As $e_0$ is symmetric, taking the determinant of the previous equation implies \begin{equation} \label{detquad} (\mu_1-\mu_0)^2\text{det}(e_0) +s^2 = 0.   
\end{equation} In order for equation (\ref{detquad}) to have solutions in the variable $s$, we must have
\begin{equation} \text{det}(e_0)\leq 0. \nonumber
\end{equation} 
\end{proof}
\begin{remark} For functions $u$ and $c$ such that the $\Gamma-$limit of $I_\epsilon$ (assuming it exists) is finite, we would expect $e(u)\in \{\mu_0,\mu_1\}e_0$. A lenient approximation of this relation is given by the hypothesis of the above proposition. A more rigorous qualification of the assumption (\ref{detcond})--in the spirit of Ball and James \cite{BallJames-FinePhase1987} or Dolzmann and M\"uller \cite{Dolzmann1995-microstructure}--is beyond our scope of interest.
\end{remark}

For a $2 \times 2$ matrix, having rank-one is equivalent to having zero determinant, and thus for symmetric $e_0$, $\text{det}(e_0)\leq 0$ holds if and only if the rank-one decompositon (\ref{R1connection2}) holds for some $\nu$. Equation (\ref{detquad}) clearly implies there are at most two possible choices of $s$, and up to sign, two choices of $\nu.$ In the following, we assume that 
\begin{equation} \label{detcond2}
\det(e_0)<0, \quad \quad  e_0\in \R^{2\times 2}_{\rm{sym}}
\end{equation} with the simpler case being that $\text{det}(e_0) = 0$ for which there is a single interface normal (see (\ref{R1connection2}) and (\ref{detquad})).

\begin{remark}
We \textit{claim} that under a change of variables, we may consider the case in which \begin{equation} \nonumber
 e_0 = e_x\otimes e_y +e_y \otimes e_x = \begin{bmatrix}
   0      & 1 \\
    1      & 0 \end{bmatrix},
 \end{equation} where we recall that $e_x$ and $e_y$ are the standard basis vectors. Note as $e_x\otimes e_y -e_y \otimes e_x $ is skew-symmetric, in this case, the normal $\nu$ in (\ref{R1connection2}) can be $\pm e_x$ or $\pm e_y$. We justify the claim: As $e_0\in \R^{2\times 2}_{\rm{sym}}$ and $\rm{det}(e_0)< 0,$ up to scaling by a diagonal matrix, there is an orthogonal matrix $\bar R$ such that  
 \begin{equation} \label{eqn:matCoV1}
 \bar R^T e_0\bar R =   \begin{bmatrix}
   -1      & 0 \\
    0      & 1 \end{bmatrix}.
 \end{equation}
In turn, direct computation shows that there is an orthogonal matrix $\bar Q$ such that 
\begin{equation}\label{eqn:matCoV2}
 \bar Q^T \begin{bmatrix}
   -1      & 0 \\
    0      & 1 \end{bmatrix} \bar Q =   \begin{bmatrix}
   0      & 1 \\
    1      & 0 \end{bmatrix} =: \tilde e_0.
    \end{equation}
 We detail how to change the energy functional $I_\epsilon$ (see (\ref{energyFunc})) assuming $e_0$ is given by the right hand side of (\ref{eqn:matCoV1}) to the form (\ref{eqn:matCoV2}); the other case, changing $e_0$ from the original matrix to the right-hand side of (\ref{eqn:matCoV1}), is similar. Define the symmetric, positive definite, fourth order tensor $\tilde{\mathbb{C}}$ by
 $$\tilde{\mathbb{C}}(v):w = \mathbb{C}(\bar Q v \bar Q^T):  (\bar Q w\bar Q^T), \quad v,w\in \R^{2\times 2}_{\rm{sym}}. $$
 For an admissible pair $(u,c)\in H^1(\Omega)\times L^2(\Omega)$ for the functional $I_\epsilon,$ we consider the transform $u \mapsto \tilde u : = \bar Q^T u(\bar Q \cdot)$ and $c \mapsto \tilde c : = c(\bar Q \cdot).$ We then define $\tilde{I}_{\epsilon}$ by (\ref{energyFunc}) with $\mathbb{C}$ and $e_0$ replaced by $\tilde{\mathbb{C}}$ and $\tilde e_0$, respectively. It follows by a change of variables that 
 $$\det(Q^T)I_\epsilon[u,c,\Omega] = \tilde I_\epsilon[\tilde u,\tilde c, \bar Q^T\Omega], $$ which justifies the claim.
  \end{remark}

\section{Compactness}\label{seccompact}

To motivate the topological convergence that we will consider for $\Gamma - $convergence, we look for appropriate function spaces where compactness holds for sequences of bounded energy.

\begin{thm} \label{thm:compactness}
Let $\Omega\subset \R^2$ be an open, bounded set with Lipschitz continuous boundary. Assume that (\ref{coercivity}) and (\ref{fcoeffCond}) hold. Let $\epsilon_i\to 0,$ $\{u_i\}_i\subset H^1(\Omega,\R^2),$ and $\{c_i\}_i \subset H^1(\Omega,[0,1])$ be such that $\sup_i I_{\epsilon_i}[u_i,c_i,\Omega]<\infty,$ where $I_\epsilon$ is the functional defined in (\ref{energyFunc}). Then up to skew-affine shifts of the functions $u_i$, we may find subsequences $\{u_{i_k}\}_k$ and $\{c_{i_k}\}_k$ with $u_{i_k}\to u$ in $H^1(\Omega,\R^2)$ and $c_{i_k} \to c$ in $L^2(\Omega)$ for some $u\in H^1(\Omega,\R^2)$ and $c\in  BV(\Omega,\{\mu_0,\mu_1\})$, such that $e(u) = ce_0.$
\end{thm}

\begin{proof}
By standard results on the Modica-Mortola (Cahn-Hilliard) functional \cite{ModicaMortola}, up to a subsequence (not relabeled), we may assume that $c_{i} \to c$ in $L^2(\Omega)$ for some $c \in BV(\Omega,\{\mu_0,\mu_1\})$. By the coercivity of the bilinear form $\mathbb{C}$ (\ref{coercivity}), we have 
$$ \int_\Omega \|e(u_{i}) - c_ie_0\|^2 \ dz \leq C\epsilon_i.$$
By the triangle inequality, $$\|e(u_i)-ce_0\|_{L^2}\leq \|e(u_i)-c_ie_0\|_{L^2}+\|c_ie_0-ce_0\|_{L^2}\to 0. $$
Define $$v_i(x,y) := u_i(x,y) -\left(\intbar_{\Omega} e(u_i(z))\ dz\right)(x,y)^T +\alpha_i,$$ where $\alpha_i$ ensures $\int_\Omega v_i \ dz = 0.$
By Korn's Inequality (see \cite{Nitsche1981-Korn}), we have $$ \|v_i\|_{H^1}\leq C\|e(v_i)\|_{L^2} = C\|e(u_i)\|_{L^2} \leq C .$$ It follows that, up to a subsequence (not relabeled), $v_i \wkto u$ in $H^1(\Omega,\R^2)$ for some $u \in H^1(\Omega,\R^2).$ By necessity, $e(u) = ce_0.$ Thus we apply Korn's inequality a second time to find $$\|v_i - u\|_{H^1}\leq C\|e(v_i - u)\|_{L^2}  = C\|e(u_i)-ce_0\|_{L^2}\to 0,$$ which proves the theorem.

\end{proof}

The above result is analogous to Theorem 2.1 in \cite{ContiSchweizer-Linear}. We note the above method of proof may be adapted to obtain the aforementioned theorem of Conti and Schweizer without the use of Young measures. The relation derived in the above compactness result, $e(u) = ce_0,$ is further characterized by the following result due to Conti and Schweizer (Proposition 2.2 in \cite{ContiSchweizer-Linear}).

\begin{thm} \label{struct} Let $\Omega\subset \R^2$ be an open, bounded set with Lipschitz continuous boundary.
Let $u\in H^1(\Omega,\R^2)$ be such that $e(u)\in BV(\Omega,\{\mu_0e_0,\mu_1e_0\}),$ where $e_0\in \R^{2\times 2}_{\rm{sym}}$ satisfies (\ref{detcond2}). Then the jumpset of $e(u)$, $J_{e(u)}$, is the union of countably many disjoint segments with constant normal and endpoints in $\partial \Omega$. Furthermore, the normal of $J_{e(u)}$ must be $\nu$ for some $\nu$ satisfying the skew symmetric rank one connection (\ref{R1connection}). Lastly, $\nabla u$ is constant in each connected component of $\Omega\setminus J_{e(u)}$. 
\end{thm}

\section{Liminf bound} \label{secliminf}
This argument is a slight variant of the one in Section 3 of \cite{ContiSchweizer-Linear}. We define the functional 
\begin{equation} \nonumber
\begin{aligned}
\mathcal{F}_{e_y}(d,l):= \inf\{\liminf\limits_{i\to \infty} & I_{\epsilon_i}[u_i,c_i,(-d,d)\times (-l,l)]:\epsilon_i\to 0,  \\
u_i\to & \bar u_{e_y} \text{ in } H^1((-d,d)\times (-l,l),\R^2),  c_i\to \bar c_{e_y} \text{ in } L^2((-d,d)\times (-l,l))\}
\end{aligned}
\end{equation} which captures the energy for a single interface in a box. Here $\bar{u}_{e_y}$ and $\bar c_{e_y}$ are defined as in (\ref{unucnu}). The proof of the following proposition is due to Fonseca and Tartar (see \cite{FonsecaTartar1989}, see also \cite{ContiFonsecaLeoni-gammConv2grad}, \cite{ContiSchweizer-Linear}).

\begin{prop} \label{FonsecaTartarProp}
Assume (\ref{coercivity}), (\ref{detcond2}), and (\ref{fcoeffCond}). Then for $d,l>0,$
\begin{equation}\label{rel1}
\mathcal{F}_{e_y}(d,l) = 2d\k(e_y), 
\end{equation} 
where $\k$ is the interfacial energy defined in (\ref{interfaceenergyconst}).
\end{prop}

\begin{proof}

For simplicity, we drop the subscript $e_y.$ To see that (\ref{rel1}) holds, we note that $\mathcal{F}(d,l)$ is a nondecreasing function of $l$. Considering sequences $\bar u_i(x) = \alpha u_i(x/\alpha)$, $\bar c_i(x) = c_i(x/\alpha)$, and $\bar \epsilon_i = \alpha\epsilon_i$, we see that 
\begin{equation}\label{scalingRelation}
\mathcal{F}(\alpha d, \alpha l) = \alpha \mathcal{F}(d,l).
\end{equation} 
By a diagonalization argument, we may find sequences $\epsilon_i$, $u_i$, and $c_i$ such that $$\mathcal{F}(d,l) = \lim\limits_{i\to \infty} I_{\epsilon_i}[u_i,c_i,(-d,d)\times (-l,l)].$$ We divide $(-d,d)$ into intervals $I_j$ of size $2d/n$ for any $n\in \N.$ For one such interval $I_j$, we must have $\liminf\limits_{i\to \infty} I_{\epsilon_i}[u_i,c_i,I_j\times (-l,l)] \leq \frac{1}{n}\mathcal{F}(d,l).$ Translating the sequence, this implies $$\mathcal{F}\left(\frac{1}{n}d,l\right)\leq \frac{1}{n}\mathcal{F}(d,l).$$ Using this inequality, letting $\alpha=1/n$ in (\ref{scalingRelation}), and by the monotonicity with respect to $l$, we conclude that $$\frac{1}{n}\mathcal{F}(d,l) = \mathcal{F}\left(\frac{1}{n}d,l\right) = \mathcal{F}\left(\frac{1}{n}d,\frac{1}{n}l\right). $$ This implies that $\mathcal{F}$ is independent of $l$, and further we have $$\mathcal{F}(d,l) = 2d\mathcal{F}(1/2,l/2d) = 2d\mathcal{F}(1/2,1/2) = 2d \k(e_y),$$ as desired.

\end{proof}

\begin{remark} \label{rmk:liminfto0}
Let $u_i\in H^{1}((-d,d)\times (-l,l),\R^2)$ and $c_i\in L^2((-d,d)\times (-l,l))$ be such that $u_i \to \bar u_{e_y}$ in $H^1,$ $c_i \to \bar c_{e_y}$ in $L^2,$ and 
\begin{equation}\nonumber
\lim\limits_{i\to \infty}I_{\epsilon_i}[u_i,c_i,(-d,d)\times (-l,l)] = 2d\k(e_y).
\end{equation} 
Then for each $0<h<l$ we have 
\begin{equation}\label{eqn:liminf0}
\lim\limits_{i\to \infty} I_{\epsilon_i}[u_i,c_i,(-d,d)\times ((-l,l)\sm (-h,h))] = 0.
\end{equation}
To see this, we apply Proposition \ref{FonsecaTartarProp} with $l$ and $h$ to find 
$$\lim\limits_{i\to \infty} I_{\epsilon_i}[u_i,c_i,(-d,d)\times (-l,l)] = 2d\k(e_y) = \mathcal{F}_{e_y}(d,l) \leq\liminf\limits_{i\to \infty} I_{\epsilon_i}[u_i,c_i,(-d,d)\times (-h,h)], $$ which implies (\ref{eqn:liminf0}).
\end{remark}

\begin{remark}
The previous proposition continues to hold if $e_y$ is replaced by a different choice of normal $\nu$ of the jumpset so that \begin{equation} \nonumber
\mathcal{F}_{\nu}(d,l) = 2d\k(\nu).
\end{equation}
\end{remark}

With this calculation in hand, we have the following theorem (see the proof of Proposition 3.1 in \cite{ContiSchweizer-Linear}). We note these results may be extended to higher dimensions relatively easily with the aid of the blow-up method (see \cite{davoli2018twowell}, \cite{Fonseca1993-LSC}, \cite{fonseca2007modern}).

\begin{thm} \label{thm:liminf}
Let $\Omega\subset \R^2$ be an open bounded set with Lipschitz continuous boundary. Assume (\ref{coercivity}), (\ref{detcond2}), and (\ref{fcoeffCond}). Then for every $u\in H^1(\Omega,\R^2)$ and $c\in L^2(\Omega),$ every $\epsilon_i \to 0$, and all $\{u_i\}_i$ in $H^1(\Omega, \R^2)$ and $\{c_i\}_i$ in $L^2(\Omega)$ with $u_i\to u$ in $H^1$ and $c_i\to c$ in $L^2,$ it holds $$ \liminf\limits_{i\to \infty}I_{\epsilon_i}[u_i,c_i,\Omega] \geq I_0[u,c,\Omega],$$ where $I_\epsilon$ and $I_0$ are defined in (\ref{energyFunc}) and (\ref{energyFuncGLim}), respectively.
\end{thm}
\begin{proof}
If $$\liminf_{i\to \infty} I_{\epsilon_i}[u_i,c_i,\Omega] = \infty, $$ then there is nothing to prove. Thus we assume the limit inferior is finite and extracting a subsequence if necessary, we may suppose that the limit inferior is a limit and $\sup_i I_{\epsilon_i}[u_i,c_i,\Omega]<\infty.$ Hence, we are in a position to apply Theorem \ref{thm:compactness} and \ref{struct} to obtain that $c\in BV(\Omega,\{\mu_0,\mu_1\})$ and $e(u) = ce_0$ and that the jumpset of $c$, $J_c,$ can be written as $$J_c = \bigsqcup_j (X_j\times \{y_j\})\sqcup \bigsqcup_j (\{x_j\}\times Y_j),$$ for some $X_j,Y_j$ intervals in $\R,$ where $\bigsqcup$ denotes a disjoint union. As $\mathcal{H}^1(J_c)<\infty,$ for any $\theta\in (0,1)$ we may find $n\in \N$ such that $$\mathcal{H}^1\Big(\bigsqcup\limits_{j=1}^n (X_j\times \{y_j\})\Big)\geq \theta\mathcal{H}^1\Big(\bigsqcup_j (X_j\times \{y_j\})\Big).$$ Scaling the intervals $X_j$, we find intervals $X_j'$ such that for all $j\leq n$, $X_j'\times \{y_j\}$ are compactly contained in $\Omega$ and $$ \mathcal{H}^1\Big(\bigsqcup\limits_{j=1}^n (X_j'\times \{y_j\})\Big)\geq \theta^2\mathcal{H}^1\Big(\bigsqcup_j (X_j\times \{y_j\})\Big).$$ Likewise we find $Y_j'.$

By Theorem \ref{struct}, the compactly contained intervals are disjoint. Furthermore, we claim there is $h>0$ such that each box $X_j'\times (y_j-h,y_j+h)$ and $(x_j-h,x_j+h)\times Y_j'$, with $j\leq n$, intersects only one interface. Let $$K := \bigsqcup\limits_{j=1}^n (X_j'\times \{y_j\})\sqcup \bigsqcup\limits_{j=1}^n (\{x_j\}\times Y_j'), \ \ H:= \bigsqcup\limits_{j=n+1}^\infty (\bar X_j\times \{y_j\})\sqcup \bigsqcup\limits_{j=n+1}^\infty (\{x_j\}\times \bar Y_j).$$ 
By Theorem \ref{struct}, we have that $\bar K$ and $H$ are disjoint. Furthermore, there cannot be $x\in \bar K \cap (\bar H\setminus H)$ as $\bar H\setminus H\subset \partial \Omega$. To see this last claim, suppose $x\in \bar H\setminus H$. Thus there must be a subsequence of distinct interfaces $\{\mathcal{I}_{j_k}\}_{k\in \N}$ such that $\mathcal{I}_{j_k} =  X_{j_k}\times \{y_{j_k}\}$ or $\mathcal{I}_{j_k} = \{x_{j_k}\}\times Y_{j_k} $ with $j_k>n$ such that $B(x,1/j_k)\cap \mathcal{I}_{j_k}\neq \emptyset.$ As the interfaces are distinct and $\mathcal{H}^1(J_c)<\infty,$ it follows $\mathcal{H}^1(\mathcal{I}_{j_k})\to 0$. Consequently, $$\text{dist}(x,\partial \Omega)\leq 1/j_k+\mathcal{H}^1(\mathcal{I}_{j_k})\to 0 $$ proving the claim. Hence the sets $\bar K$ and $\bar H$ are disjoint, which shows that such an $h$ exists.

Using Proposition \ref{FonsecaTartarProp}, we find 
\begin{align*}
\liminf\limits_{i\to \infty} & \ I_{\epsilon_i}[u_i,c_i,\Omega]\\
\geq & \sum\limits_{i=1}^n \liminf\limits_{i\to \infty}\Big(I_{\epsilon_i}[u_i,c_i,X_j'\times (y_j-h,y_j+h)]+I_{\epsilon_i}[u_i,c_i,(x_j-h,x_j+h)\times Y_j']\Big)\\
\geq & \sum\limits_{i=1}^n (\mathcal{L}^1(X_j')k(e_y)+\mathcal{L}^1(Y_j')k(e_x)) \geq  \theta^2 \int_{J_c} k(\nu) \ d \mathcal{H}^1.
\end{align*}
Letting $\theta\to 1,$ we complete the proof.
\end{proof}

\section{Characterization of interfacial energy}\label{sec:charInterface} 
In this section, we characterize the interfacial energy on a box in terms of $\k (e_y),$ defined in (\ref{interfaceenergyconst}), via the following theorem.
\begin{thm}\label{thm:charInterface}
Let $\epsilon_i\to 0,$ $l>0$, and $d>0$. There exists sequences $u_i\to  \bar u_{e_y}$ in $H^1((-d/2,d/2)\times (-l,l), \R^2)$ and $c_i\to \bar c_{e_y}$ in $L^2((-d/2,d/2)\times (-l,l))$ such that 
\begin{equation}\label{lim:charInterface}
\lim\limits_{i\to \infty} I_{{ \epsilon_i}}[ u_i, c_i,(-d/2,d/2)\times (-l,l)] = d\k(e_y).
\end{equation} 
Furthermore, $\bar c_i = \bar c$ and $\bar u_i = \bar u +\chi_{y<0}(R_{\phi_i}(x,y)^T+a_i)$ in some neighborhood of the upper and lower boundaries $\{(x,y)\in (-d/2,d/2)\times\R : y = \pm l\}$, where $|\phi_i|+|a_i|\to 0,$ and $R_\phi$ is defined in (\ref{def:skewsymRot}).
\end{thm}

To motivate the criticality of the above theorem, when proving the $\limsup$ bound, we will need to construct a minimizing sequence of functions for a relatively generic domain. To construct such a sequence, we will interpolate between minimizing sequences for boxes containing a single interface. Accepting that this will be the applied methodology, a theorem like the above is crucial to interpolation. We note however that there are other possible methods including proof of an $H^{1/2}$ bound for a general domain or box (see Theorem \ref{H1/2bound} and \cite{davoli2018twowell}).

As the proof of Theorem \ref{thm:charInterface} is involved, we decompose it into three steps.

\begin{itemize}
\item[\textbf{Step I}] \label{stepii1}
Suppose $$\lim_i I_{\epsilon_i}[u_i,c_i,(-2d,2d)\times(-l,l)]=4d\k(e_y),$$ with $u_i\to\bar u_{e_y}$ and $c_i\to \bar c_{e_y}$. We will find new sequences $\bar u_i \to \bar u_{e_y}$ and $\bar c_i \to \bar c_{e_y}$ such that $$\limsup_i I_{\epsilon_i}[\bar u_i,\bar c_i,(-d/2,d/2)\times(-l,l)]\leq d\k(e_y).$$ Furthermore both $\bar c_i = \bar c_{e_y}$ and $\bar u_i = \bar u_{e_y} +(R_{\phi_i}(x,y)^T+a_i)\chi_{y<0}$ in some neighborhood of the upper and lower boundaries $\{(x,y)\in (-d/2,d/2)\times\R : y = \pm l\}$, where $|\phi_i|+|a_i|\to 0.$ See Theorem \ref{boxenergy}.

\item[\textbf{Step II}] \label{stepii2}
Let $\epsilon_i\to 0,$ $l>0$, and $d>0$. There exists sequences $u_i\to \bar u_{e_y}$ and $c_i\to \bar c_{e_y}$ such that 
\begin{equation}\nonumber
\lim\limits_{i\to \infty} I_{{ \epsilon_i}}[ u_i, c_i,(-d,d)\times(-l,l)] = 2d\k(e_y).
\end{equation} 
See Theorem \ref{boxseq}.

\item[\textbf{Step III}]\label{stepii3} We bring together the previous two steps to complete the proof of Theorem \ref{thm:charInterface}.

\end{itemize}

\subsection*{Proof of Step I}
In the following we fix $l>0$ and for $d>0$ and $\epsilon_i \in \R$ let 
\begin{equation}\label{def:DBox}
\begin{aligned}
 D_{d} &:= (-d,d)\times (-l,l),\quad D_{d,\epsilon_i} := \{(x,y)\in D_{d}:y_i\leq y\leq y_i+\epsilon_i\},\\
D_{d,\epsilon_i}^- &:= \{(x,y)\in D_{d}:y< y_i\},\quad D_{d,\epsilon_i}^+ := \{(x,y)\in D_{d}: y_i+\epsilon_i<y\}.
\end{aligned}
\end{equation}

\begin{thm}\label{boxenergy}
Let $d>0.$ Assume that (\ref{coercivity}), (\ref{detcond2}), and (\ref{fcoeffCond}) hold, and suppose
\begin{equation} \label{hyp:boxenergy}
\lim_i I_{\epsilon_i}[u_i,c_i,D_{2d}]=4d\k(e_y),
\end{equation} with $u_i\to \bar u_{e_y}$ in $H^1(D_{2d},\R^2)$ and $c_i\to \bar c_{e_y}$ in $L^2(D_{2d}),$ where $\k(e_y)$ and $\bar{u}_{e_y}$ are defined in (\ref{interfaceenergyconst}) and (\ref{unucnu}) respectively. We may find new sequences $\bar u_i \to \bar u_{e_y}$ and $\bar c_i \to \bar c_{e_y}$ in the same respective spaces such that $$\lim_i I_{\epsilon_i}[\bar u_i,\bar c_i,D_{d/2}]= d\k(e_y).$$ Furthermore both $\bar c_i = \bar c_{e_y}$ and $\bar u_i = \bar u_{e_y} +(R_{\phi_i}(x,y)^T+a_i)\chi_{\{y<0\}}$ in some neighborhood of the upper and lower boundaries of $D_{2d}$, where $|\phi_i|+|a_i|\to 0.$  
\end{thm}

\begin{remark} \label{rmk:poincareChallenge}
A standard approach to proving this type of theorem (for the top boundary) for first order Cahn-Hilliard functionals would involve sequences as given by the following: Let $\psi:\R\to [0,1]$ be a smooth cutoff function with $\psi(x) =1$ for $x<0$ and $\psi(x)=0$ for $x>1$. For some $y_i\in (l/4,3l/4)$ to be determined, let $\psi_i(x,y) := \psi((y-y_i)/\epsilon_i)$ and define 
\begin{align*}
& \bar u_i := \psi_i\Big(u_i -\intbar_{D_{2d,\epsilon_i}}(u_i-\bar u_{e_y}) \ dz \Big)+(1-\psi_i)\bar u_{e_y},\\
& \bar c_i := \psi_i c_i+(1-\psi_i)\bar c_{e_y}.
\end{align*}
Analyzing the energy, it turns out that the elastic energy presents the main difficulty, wherein we have an energy term of the form
\begin{align*}
\int_{D_{2d,\epsilon_i}}  \frac{1}{\epsilon_i}\left\|(u_i-\bar u_{e_y}- \intbar_{D_{2d,\epsilon_i}}(u_i-\bar u_{e_y}) \ dw )\otimes \nabla \psi_i\right\|^2 \ dz  & \approx  \\
 \int_{D_{2d,\epsilon_i}}  \frac{1}{\epsilon_i^3} & \left\|(u_i-\bar u_{e_y}- \intbar_{D_{2d,\epsilon_i}}(u_i-\bar u_{e_y}) \ dw )\right\|^2 \ dz.
\end{align*}
Here we see that the mean subtraction was introduced in hopes that the Poincar\'e inequality (see \cite{LeoniBook}) might suffice to bound the term. However, with this we have 

$$\int_{D_{2d,\epsilon_i}}  \frac{1}{\epsilon_i^3}\left\|(u_i-\bar u_{e_y}- \intbar_{D_{2d,\epsilon_i}}(u_i-\bar u_{e_y}) \ dw)\right\|^2 \ dz \leq \int_{D_{2d,\epsilon_i}}  \frac{\max\{\epsilon_i,d\}^2}{\epsilon_i^3}\|\nabla (u_i-\bar u_{e_y})\|^2 \ dz, $$
which cannot be controlled via averages as $\epsilon_i<d$ for large $i$. Consequently, it is crucial that we apply the Poincar\'e inequality for $H^1_0$, in some sense, which will replace the maximum in the above inequality with $\epsilon_i$ itself.
\end{remark}
To prove Theorem \ref{boxenergy} and overcome the challenges posed by Remark \ref{rmk:poincareChallenge}, we derive an $H^{1/2}$ bound for low energy functions which will help to control the trace of $u$ on $D_{2d,\epsilon_i}$. The proof relies on ideas of Conti and Schweizer (see Section 4 of \cite{ContiSchweizer-Linear}) who derive an analogous bound for functionals of the form $F_\epsilon$ (see (\ref{csEnergyFunc})), as mentioned in the introduction.

We prove a lemma which allows us to control some energies via averages.

\begin{lem}\label{energylemma}
Let $\eta>0$. Supposing $r:[a,b]\to[0,\infty)$ is an integrable function with $\int_{a}^{b}r\ dx\leq \eta,$ then for any $\theta\in (0,1)$ there exists a measurable set $E_\theta\subset [a,b]$ with measure at least $\theta(b-a)$ such that $$r\leq \frac{\eta}{(1-\theta)(b-a)} \quad \text{ on } E_\theta. $$ 
\end{lem}

\begin{proof}
Proceeding by contradiction, we have $\mathcal{L}^1(\{r\leq \frac{\eta}{(1-\theta)(b-a)}\})<\theta (b-a).$ Thus $\mathcal{L}^1(\{r> \frac{\eta}{(1-\theta)(b-a)}\})\geq (1-\theta)(b-a),$ which implies that $\int_a^b r \ dx>\eta,$ a contradiction.
\end{proof}

\begin{thm}\label{H1/2bound}
Assume (\ref{coercivity}), (\ref{detcond2}), and (\ref{fcoeffCond}) hold. Given $d>0,l_1>l_0$, $c\in H^1((-d,d)\times (l_0,l_1))$, and $u\in C^2((-d,d)\times (l_0,l_1), \R^2)$, there are constants $\eta_0,C>0$ such that if $(\zeta_u,\zeta_c)\in\{(\mu_0e_0,\mu_0),(\mu_1e_0,\mu_1)\},$ $$I_\epsilon[u,c,(-d,d)\times (l_0,l_1)]\leq \eta \leq \eta_0,$$ and $$\|e(u) - \zeta_u\|^2_{L^2((-d,d)\times (l_0,l_1))}+\|c-\zeta_c\|^2_{L^2((-d,d)\times (l_0,l_1))}\leq \eta,$$
then for some set $E\subset (l_0,l_1)$ with $\mathcal{L}^1(E)>\frac{l_1-l_0}{2},$ we have the following: 
For all $y\in E$ there is an affine function $w_y:\R^2\to \R^2$ with $e(w_y) = \zeta_u$ such that $$\|u- w_y\|^2_{H^{1/2}((-d/2,d/2)\times \{y\})} \leq C\eta\epsilon.$$
\end{thm}

To prove this, $H^{1/2}$ bound, we are immediately drawn to looking at the elastic energy which heuristically looks like $$\int_{D_d} \frac{1}{\epsilon}\min\{\|e(u)-\mu_0 e_0\|,\|e(u)-\mu_1 e_0\|\}^2 \ dz .$$ If we could simply conclude that $\|e(u)-\mu_1e_0\|\leq \|e(u)-\mu_0 e_0\|$ in $D_d,$ we could then apply Korn's Inequality to conclude $\|u-w\|_{H^1}^2\leq C\eta \epsilon,$ where $e(w) = \mu_1 e_0.$ From which we could apply standard trace bounds to conclude the theorem. But to conclude the pointwise estimate $\|e(u)-\mu_1 e_0\|\leq \|e(u) - \mu_0 e_0\|$ appears infeasible. Thus we proceed via the methods of Conti and Schweizer (Section 4 of \cite{ContiSchweizer-Linear}), wherein we find a large set $E\subset (-l,l)$ for which we may define some function $\bar u_y$ associated to each $y\in E$ which satisfies $\bar u_y(\cdot,y) = u(\cdot,y)$ and has energy estimates representative of $\|e(\bar u_y)-\mu_1 e_0\|\leq \|e(\bar u_y) - \mu_0 e_0 \|,$ consequently reducing the problem to an application of Korn's inequality. Finding the function $\bar u_y$ involves nontrivial constructions, and will be constructed via linear interpolations of averages of $u$ on a grid which refines towards the line $(-d,d)\times \{y\}.$

\subsection*{Grid Energy estimates}

We define 
\begin{equation}\label{G1grid}
G^1 := \{(x,y):(x,y)\in \partial (0,1)^2 \text{ or } x=y \text{ or }x = 1-y\}.
\end{equation}

For some fixed $n\in \N$, we then set 
\begin{equation}\label{Gngrid}
G^n := \bigcup_{i,j=0}^{n-1} \Big( (i/n,j/n)+\frac{1}{n}G^1\Big).
\end{equation}

\begin{multicols}{2}
\begin{figure}[H]
\begin{center}
\begin{tikzpicture}
\draw (0,0) -- (0,4) -- (4,4) -- (4,0) -- (0,0);
\draw (0,0) -- (4,4);
\draw (4,0) -- (0,4);
\end{tikzpicture}
\end{center}
 \caption{\textbf{\small{$G^1$, see (\ref{G1grid}). }}}
 \end{figure}

\begin{figure}[H]
\begin{center}
\begin{tikzpicture}
\draw (0,0) -- (0,2) -- (2,2) -- (2,0) -- (0,0);
\draw (0,0) -- (2,2);
\draw (2,0) -- (0,2);
\draw (2,0) -- (2,4) -- (4,4) -- (4,2) -- (2,2);
\draw (2,2) -- (4,4);
\draw (4,2) -- (2,4);
\draw (0,2) -- (0,4) -- (2,4) -- (2,2) -- (0,2);
\draw (0,2) -- (2,4);
\draw (2,2) -- (0,4);
\draw (2,0) -- (2,2) -- (4,2) -- (4,0) -- (2,0);
\draw (2,0) -- (4,2);
\draw (4,0) -- (2,2);
\end{tikzpicture}
\end{center}
 \caption{\textbf{\small{$G^2$, see (\ref{Gngrid}). }}}
 \end{figure}
\end{multicols}

For some fixed $k\in \N$, we define $d_k := 2^{-k}$ and suppose $z=(x,y),z'=(x',y')\in \R^2$ (with $y<y'$) are the left vertices of a parallelogram $P$ with a base of length $d_k$ parallel to the $x$-axis; consider the affine map $L_k(z,z'):\R^2\to \R^2$ which maps $(0,1)^2$ onto $P$ with $L_k(z,z')(0,0) = z$ and $L_k(z,z')(0,1) = z'.$

We define 
\begin{equation}\label{Gnkgrid}
G^n_k(z,z') := L_k(z,z')\Big[\bigcup_{i=0}^{\Delta 2^{k}-1}((i,0)+G^n)\Big],
\end{equation} where $\Delta >0$ is such that $\Delta 2^k$ is an integer. 
\begin{figure}[H]
\begin{center}
\begin{tikzpicture}
\draw (0,0) -- (1,4) -- (5,4) -- (4,0) -- (0,0);
\draw (0,0) -- (5,4);
\draw (4,0) -- (1,4);
\draw (4,0) -- (5,4) -- (9,4) -- (8,0) -- (4,0);
\draw (4,0) -- (9,4);
\draw (8,0) -- (5,4);
\draw (8,0) -- (9,4) -- (13,4) -- (12,0) -- (8,0);
\draw (8,0) -- (13,4);
\draw (12,0) -- (9,4);
\draw (12,0) -- (13,4) -- (17,4) -- (16,0) -- (12,0);
\draw (12,0) -- (17,4);
\draw (16,0) -- (13,4);
\end{tikzpicture}
\end{center}
 \caption{\textbf{\small{$G^1_2(z,z')$ for $\Delta =1,$ $z = (0,0),$ $z' = (1/4,1)$, see (\ref{Gnkgrid}). }}}
 \end{figure}
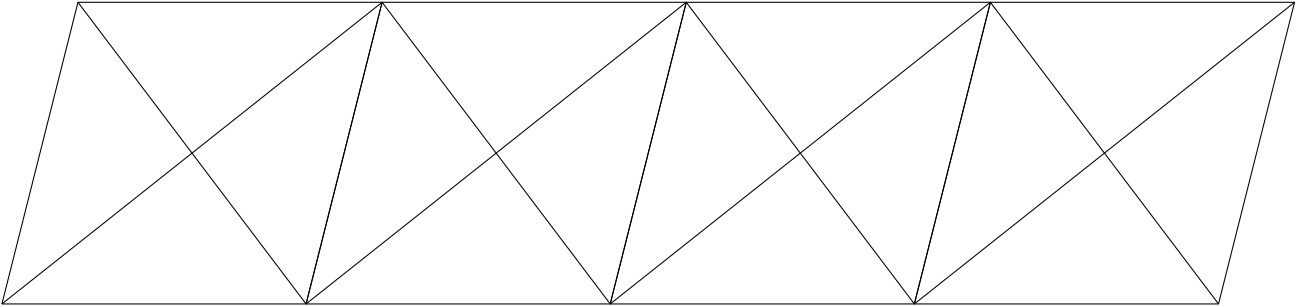

Let \begin{equation}\label{def:g_eps}
g_\epsilon(x,y) := \frac{1}{\epsilon}f(c(x,y)) + \epsilon \|\nabla c(x,y)\|^2 +\frac{1}{\epsilon}\|e(u(x,y))-c(x,y)e_0\|^2.
\end{equation} 
Up to modification of a few constants, the proof of the following theorem follows closely the one of Lemma 4.3 in \cite{ContiSchweizer-Linear}, and hence we refer the reader to this for a proof.
\begin{thm}\label{gridenergythm}
Assume (\ref{coercivity}), (\ref{detcond2}), and (\ref{fcoeffCond}) hold. Given $\theta\in (0,1)$, $\delta\in (0,1/4),$ $d>0$, $l_1>l_0$, and $(\zeta_u,\zeta_c)\in\{(\mu_0e_0,\mu_0),(\mu_1e_0,\mu_1)\}$, there are constants $\eta_0,\epsilon_0,C,k_0,\Delta, C_{d,l}>0$ such that for all $\epsilon\in (0,\epsilon_0)$, $u\in C^2((-d,d)\times (l_0,l_1),\R^2)$, $c\in C^1((-d,d)\times (l_0,l_1),[0,1])$ satisfying $$ I_\epsilon[u,c,(-d,d)\times (l_0,l_1)]\leq \eta \leq \eta_0$$
and 
$$\|e(u) - \zeta_u\|^2_{L^2((-d,d)\times (l_0,l_1))}+\|c-\zeta_c\|^2_{L^2((-d,d)\times (l_0,l_1))}\leq \eta, $$
we may find a set $E\subset(l_0,l_1)$ with $\mathcal{L}^1(E)>\frac{l_1-l_0}{2}$ for which we have the following:
For each $y_0\in E$, $k>k_0$
\begin{enumerate}[label=\roman*), ref=\roman*]
\item There is $z_k = (x_k,y_k)$ with $y_k\in [y_0-d_{k-1},y_0-d_{k-1}+\delta d_{k-1}]$ and $|x_k-x_{k+1}|\leq \delta  d_k,$ and $-x_k\in (-d,-d+3\delta).$
\item $I_\epsilon[u,c,(-d,d)\times (y_k,y_0)]\leq C\eta |y_0-y_k|.$
\item\label{deltabdd} For all points $z$ in the grid $G^n_k(z_k,z_{k+1})$ defined in (\ref{Gnkgrid}), $|c(z)-\zeta_c|\leq \delta$.
\item\label{gridenergy} We have the energetic bound $$\int_{G^n_k(z_k,z_{k+1})} g_{\epsilon} \ d\mathcal{H}^1\leq C\eta, $$ where $g_\epsilon$ is defined in (\ref{def:g_eps}).
\item $\Delta 2^{k_0}\in \N$ and $(-d/2,d/2)\times (y_0-C_{d,l},y_0)$ is contained in $\bigcup_{k>k_0}\operatorname{conv}(G^n_k(z_k,z_{k+1})).$
\end{enumerate}
\end{thm}

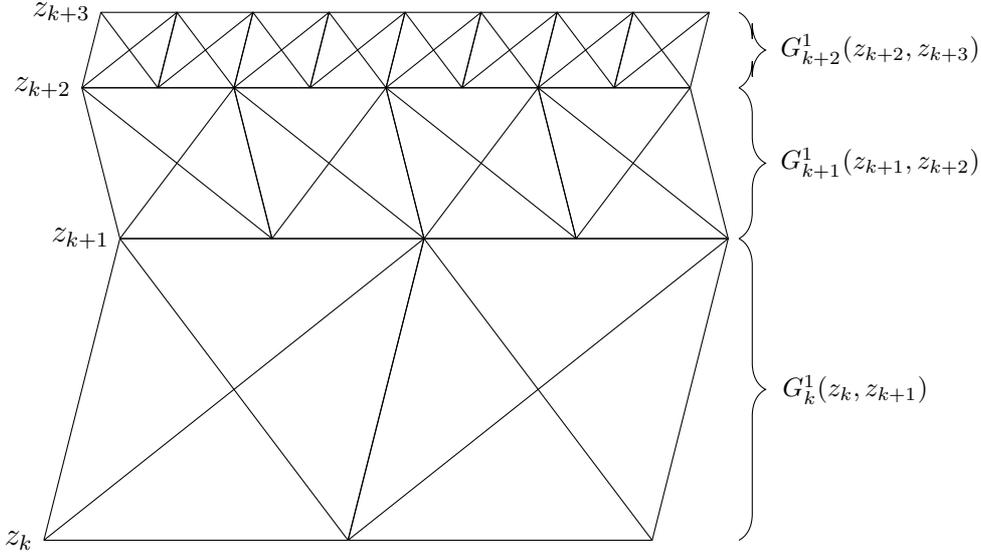
\begin{figure}[H]
\begin{center}
\begin{tikzpicture}
\draw (0,0) node[anchor= east] {$z_{k}$} -- (1,4) -- (5,4) -- (4,0) -- (0,0);
\draw (0,0) -- (5,4);
\draw (4,0) -- (1,4);
\draw (4,0) -- (5,4) -- (9,4) -- (8,0) -- (4,0);
\draw (4,0) -- (9,4);
\draw (8,0) -- (5,4);
\draw [decorate,decoration={brace,amplitude=10pt,mirror},xshift=4pt,yshift=0pt]
(9,0) -- (9,4) node [black,midway,xshift=1.53cm] 
{\footnotesize $G^1_k(z_k,z_{k+1})$};
\draw (1,4) node[anchor= east] {$z_{k+1}$} -- (3,4) -- (2.5,6) -- (0.5,6) -- (1,4);
\draw (1,4) -- (2.5,6);
\draw (3,4) -- (0.5,6);
\draw (3,4) -- (5,4) -- (4.5,6) -- (2.5,6) -- (3,4);
\draw (3,4) -- (4.5,6);
\draw (5,4) -- (2.5,6);
\draw (5,4) -- (7,4) -- (6.5,6) -- (4.5,6) -- (5,4);
\draw (5,4) -- (6.5,6);
\draw (7,4) -- (4.5,6);
\draw (7,4) -- (9,4) -- (8.5,6) -- (6.5,6) -- (7,4);
\draw (7,4) -- (8.5,6);
\draw (9,4) -- (6.5,6);
\draw [decorate,decoration={brace,amplitude=10pt,mirror},xshift=4pt,yshift=0pt]
(9,4) -- (9,6) node [black,midway,xshift=1.86cm] 
{\footnotesize $G^1_{k+1}(z_{k+1},z_{k+2})$};
\draw (0.5,6) node[anchor= east] {$z_{k+2}$} -- (1.5,6) -- (1.75,7) -- (0.75,7) node[anchor= east] {$z_{k+3}$} -- (0.5,6);
\draw (0.5,6) -- (1.75,7);
\draw (1.5,6) -- (0.75,7);
\draw (1.5,6) -- (2.5,6) -- (2.75,7) -- (1.75,7) -- (1.5,6);
\draw (1.5,6) -- (2.75,7);
\draw (2.5,6) -- (1.75,7);
\draw (2.5,6) -- (3.5,6) -- (3.75,7) -- (2.75,7) -- (2.5,6);
\draw (2.5,6) -- (3.75,7);
\draw (3.5,6) -- (2.75,7);
\draw (3.5,6) -- (4.5,6) -- (4.75,7) -- (3.75,7) -- (3.5,6);
\draw (3.5,6) -- (4.75,7);
\draw (4.5,6) -- (3.75,7);
\draw (4.5,6) -- (5.5,6) -- (5.75,7) -- (4.75,7) -- (4.5,6);
\draw (4.5,6) -- (5.75,7);
\draw (5.5,6) -- (4.75,7);
\draw (5.5,6) -- (6.5,6) -- (6.75,7) -- (5.75,7) -- (5.5,6);
\draw (5.5,6) -- (6.75,7);
\draw (6.5,6) -- (5.75,7);
\draw (6.5,6) -- (7.5,6) -- (7.75,7) -- (6.75,7) -- (6.5,6);
\draw (6.5,6) -- (7.75,7);
\draw (7.5,6) -- (6.75,7);
\draw (7.5,6) -- (8.5,6) -- (8.75,7) -- (7.75,7) -- (7.5,6);
\draw (7.5,6) -- (8.75,7);
\draw (8.5,6) -- (7.75,7);
\draw [decorate,decoration={brace,amplitude=10pt,mirror},xshift=4pt,yshift=0pt]
(9,6) -- (9,7) node [black,midway,xshift=1.86cm] 
{\footnotesize $G^1_{k+2}(z_{k+2},z_{k+3})$};
\end{tikzpicture}
\end{center}
 \caption{\textbf{\small{This figure illustrates the collection of grids constructed in Theorem \ref{gridenergythm} in the case that $n=1$. }}}
 \end{figure}

Without loss of generality, suppose $(\zeta_u,\zeta_c) = (\mu_0e_0,\mu_0)$. Utilizing properties \ref{deltabdd} and \ref{gridenergy} in Theorem \ref{gridenergythm} and that $f$ is super-quadratic at the wells (see Proposition \ref{prop:ffunc}), we find that $$\int_{G^n_k(z_k,z_{k+1})} |c-\mu_0|^2 \ d\mathcal{H}^1\leq C\eta\epsilon,$$ which by Minkowski's inequality (see \cite{fonseca2007modern}) and property \ref{gridenergy} in Theorem \ref{gridenergythm} allows us to further conclude 
\begin{equation} \label{elasticgridenergy}
\int_{G^n_k(z_k,z_{k+1})} \|e(u)-\mu_0e_0\|^2 \ d\mathcal{H}^1\leq C\eta\epsilon.
\end{equation}

We include a lemma of Conti and Schweizer \cite{ContiSchweizer-Linear} relating energy bounds on one element of the grid to an affine approximation of the function $u$. Let
\begin{equation}\label{Lmatrix}
L := \begin{bmatrix}
1/l & s \\
0 & l 
\end{bmatrix}
\end{equation} be the matrix mapping the unit square onto the parallelogram with vertices $(0,0), \ (1/l,0), \ (s,l),$ and $(s+1/l,l)$. For all $s,l$ with $|s|+|l-1|$ sufficiently small, the parallelogram is ``close" to the square.

Letting $a\in \R^2$, $s^- := 0$, $s^+ := s,$ $l^- := 0$, and $l^+ := l,$ we define (see Figure \ref{fig:BoxGrid}) the segments $\gamma_i^{\pm}$ on the grid given by $a+ L(dG^{n})$ as  
\begin{equation} \label{def:segmentGamma}
\gamma_{i}^{\pm} : =a+ \Big( (s^{\pm}d+\frac{i}{n}(d/l),s^{\pm}d+\frac{i+1}{n}(d/l))\times \{ dl^{\pm}\} \Big),
\end{equation}
with left endpoints $z_i^{\pm}$ given by 
\begin{equation} \label{def:znode}
z_{i}^{\pm} : = a+ (s^{\pm}d+\frac{i}{n}(d/l),  dl^{\pm}).
\end{equation}

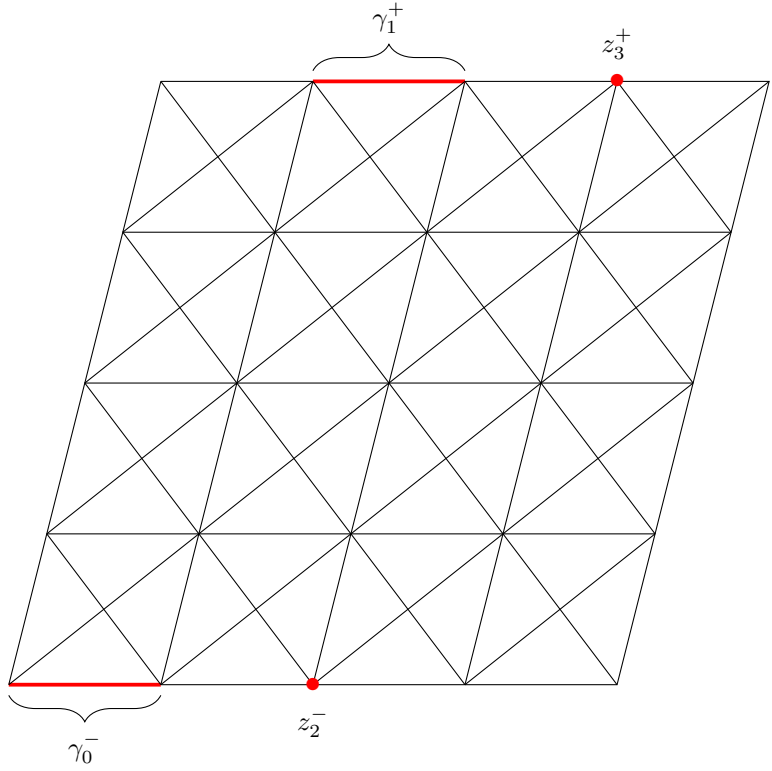
\begin{figure}[t]
\begin{center}
\begin{tikzpicture}
\draw [line width=0.5mm, red ]  (0,0) -- (2,0);
\draw  (2,0) -- (2.5,2) -- (0.5,2) --(0,0);
\draw (0,0) -- (2.5,2);
\draw (2,0) -- (0.5,2);
\draw [decorate,decoration={brace,amplitude=10pt,mirror},xshift=0pt,yshift=-4pt]
(0,0) -- (2,0) node [black,midway,xshift=0.0cm,yshift=-0.7cm] 
{\footnotesize $\gamma_0^-$};
\draw (2,0) -- (4,0) -- (4.5,2) -- (2.5,2);
\draw (2,0) -- (4.5,2);
\draw (4,0) -- (2.5,2);
\draw (4,0) -- (6,0) -- (6.5,2) -- (4.5,2);
\draw (4,0) -- (6.5,2);
\draw (6,0) -- (4.5,2);
\draw (6,0) -- (8,0) -- (8.5,2) -- (6.5,2);
\draw (6,0) -- (8.5,2);
\draw (8,0) -- (6.5,2);
\node [red] at (4,0) {\textbullet};
\node at (4,0) [black,xshift=0.0cm,yshift=-0.5cm] {\footnotesize $z_2^-$};
\draw (2.5,2) -- (3,4) -- (1,4) --(0.5,2);
\draw (0.5,2) -- (3,4);
\draw (2.5,2) -- (1,4);
\draw (4.5,2) -- (5,4) -- (3,4);
\draw (2.5,2) -- (5,4);
\draw (4.5,2) -- (3,4);
\draw (6.5,2) -- (7,4) -- (5,4);
\draw (4.5,2) -- (7,4);
\draw (6.5,2) -- (5,4);
\draw (8.5,2) -- (9,4) -- (7,4);
\draw (6.5,2) -- (9,4);
\draw (8.5,2) -- (7,4);
\draw (3,4) -- (3.5,6) -- (1.5,6) --(1,4);
\draw (1,4) -- (3.5,6);
\draw (3,4) -- (1.5,6);
\draw (5,4) -- (5.5,6) -- (3.5,6);
\draw (3,4) -- (5.5,6);
\draw (5,4) -- (3.5,6);
\draw (7,4) -- (7.5,6) -- (5.5,6);
\draw (5,4) -- (7.5,6);
\draw (7,4) -- (5.5,6);
\draw (9,4) -- (9.5,6) -- (7.5,6);
\draw (7,4) -- (9.5,6);
\draw (9,4) -- (7.5,6);
\draw [decorate,decoration={brace,amplitude=10pt},xshift=0pt,yshift=4pt]
(4,8) -- (6,8) node [black,midway,xshift=0.0cm,yshift=0.7cm] 
{\footnotesize $\gamma_1^+$};

\draw (3.5,6) -- (4,8);
\draw (4,8) -- (2,8);
\draw (2,8) --(1.5,6);
\draw (1.5,6) -- (4,8);
\draw (3.5,6) -- (2,8);
\draw (5.5,6) -- (6,8);
\draw [line width=0.5mm, red ] (6,8) -- (4,8);
\draw (3.5,6) -- (6,8);
\draw (5.5,6) -- (4,8);
\draw (7.5,6) -- (8,8) -- (6,8);
\draw (5.5,6) -- (8,8);
\draw (7.5,6) -- (6,8);
\draw (9.5,6) -- (10,8) -- (8,8);
\draw (7.5,6) -- (10,8);
\draw (9.5,6) -- (8,8);
\node [red] at (8,8) {\textbullet};
\node at (8,8) [black,xshift=0.0cm,yshift=0.5cm] {\footnotesize $z_3^+$};
\end{tikzpicture}
\end{center}
 \caption{\textbf{\small{Grid $L(dG^4)$ with segments $\gamma_{i}^{\pm}$, see (\ref{def:segmentGamma}), and points $z_i^\pm$, see (\ref{def:znode}). }}}
  \label{fig:BoxGrid}
 \end{figure}

Across all parallelograms sufficiently close to the square, we have the following affine approximation result: 

\begin{lem}\label{conti4.4}(Lemma 4.4, Remark 4.5 in \cite{ContiSchweizer-Linear}) Suppose $a\in \R^2$, $d>0,$ and $\zeta_u\in\{\mu_0e_0,\mu_1e_0\}$. There exist constants $\delta, t_0,C>0$ such that for all $s,l,$ with 
\begin{equation}\label{slCondition}
|s|+|l-1|<\delta,
\end{equation} and $u\in H^1(a+L(0,d)^2,\R^2),$ with  
$$ \frac{1}{d^2}\int_{a+L(0,d)^2} \min\{\|e(u)-\mu_0e_0\|^2,\|e(u)- \mu_1e_0\|^2\} \ dz \leq \sigma$$
and
$$ \frac{1}{d}\int_{a+L(dG^n)} \|e(u)-\zeta_u\|^2\ d\mathcal{H}^1\leq \sigma,$$
we may find $\phi\in \R$ and $w_0\in \R^2$ such that for $i=0,\ldots, n-1,$ 
$$u_i^{\pm}:=\intbar_{\gamma_i^{\pm}}u \ d\mathcal{H}^1$$
and
$$ w_i^\pm : = w_0+\zeta_u (z^{\pm}_i)+R_\phi (z^{\pm}_i),$$
we have $$\|u_i^\pm -w_i^\pm\|^2\leq C\sigma d^2.$$
We recall that $R_\phi$, $G^n$, and $L$ are defined in (\ref{def:skewsymRot}), (\ref{Gngrid}), and (\ref{Lmatrix}) respectively. Furthermore, $\gamma_{i}^{\pm}$ and $z_i^{\pm}$ are depicted in Figure \ref{fig:BoxGrid}.
\end{lem}

To obtain the $H^{1/2}$ bound in Theorem \ref{H1/2bound}, it is essential that we estimate how $\phi$ changes between neighboring parallelograms. We collect these estimates in the following lemma.

\begin{lem} \label{lem:adjacentBoxes}
Suppose $n=4$, $a\in \R^2,$ $Q_0 = L_0[a+(0,d)^2], $ and one of the following cases

\textit{Case 1: } $Q_1 = L_1[a+ (0,d) + (0,\frac{1}{2}d) \times (0,\frac{1}{2}d)],$  

\textit{Case 2: } $Q_1 = L_0[a+ (d,0) + (0,d)\times(0,d)],$

\textit{Case 3: } $Q_1 = L_0[a+ (\frac{1}{2}d,0) +(0,d)\times(0,d)],$

\noindent where $L_0$ and $L_1$ are affine maps with linear part of the form (\ref{Lmatrix}) with parameters $l_i,s_i,$ subindexed by $0$ and $1$ respectively, satisfying condition (\ref{slCondition}) of Lemma \ref{conti4.4}. We further assume that $L_0(0,d) = L_1(0,d)$ and $L_0(d,d) = L_1(d,d).$ Then if $u\in H^1((\ov{Q_0\cup Q_1})^{\mathrm{o}},\R^2)$, we have that parameters $\phi_0$ and $w_{0,0}$ associated to the grid $P_0 = L_0(a+dG^4)$ and parameters $\phi_1$ and $w_{0,1}$ associated to the grid 

\textit{Case 1: } $P_1 = L_1(a+(0,d)+\frac{1}{2}dG^4),$

\textit{Case 2: } $P_1 = L_0(a+(d,0)+dG^4),$

\textit{Case 3: } $P_1 = L_0(a+(0,\frac{1}{2}d)+dG^4),$

\noindent by applications of Lemma \ref{conti4.4} satisfy the bounds
\begin{equation} \nonumber
\|w_{0,0}-w_{0,1}\|\leq C \sqrt{\sigma} d
\end{equation}
and
\begin{equation} \nonumber
|\phi_{0}-\phi_{1}\|\leq C \sqrt{\sigma},
\end{equation}
where 
$$\sigma:= \frac{1}{d^2}\int_{Q_0\cup Q_1} \min\{\|e(u)\|^2,\|e(u)-e_0\|^2\} \ dz  + \frac{1}{d}\int_{P_0\cup P_1} \|e(u)\|^2\ d\mathcal{H}^1.$$
\end{lem}

\begin{figure*}[t!]
    \centering
    \begin{subfigure}[t]{0.5\textwidth}
        \centering
        \begin{tikzpicture}
\node at (2,2) {$Q_0$};
\draw (0,0) -- (0,4) -- (4,4) -- (4,0) -- (0,0);
\node at (1,5) {$Q_1$};
\draw (0,4) -- (2,4) -- (2,6) -- (0,6) -- (0,4);
\end{tikzpicture}
        \caption{\textbf{\small{Case 1. }}}
    \end{subfigure}%
    ~ 
    \begin{subfigure}[t]{0.5\textwidth}
        \centering
        \begin{tikzpicture}
\node at (2,2) {$Q_0$};
\draw (0,0) -- (0,4) -- (4,4) -- (4,0) -- (0,0);
\node at (6,2) {$Q_1$};
\draw (4,0) -- (8,0) -- (8,4) -- (4,4) -- (4,0);
\end{tikzpicture}
       \caption{\textbf{\small{Case 2. }}}
    \end{subfigure}\par\medskip
    \begin{subfigure}[t]{0.5\textwidth}
\begin{tikzpicture}
\node at (1,2) {$Q_0$};
\draw (0,0) -- (0,4) -- (4,4) -- (4,0) -- (0,0);
\node at (5,2) {$Q_1$};
\draw [dashed] (2,0) -- (6,0) -- (6,4) -- (2,4) -- (2,0);
\end{tikzpicture}
 \caption{\textbf{\small{Case 3. }}}
  \end{subfigure}
    
  \caption{\textbf{\small{Cases of Lemma \ref{lem:adjacentBoxes} when $L=I$. }}}
\end{figure*}
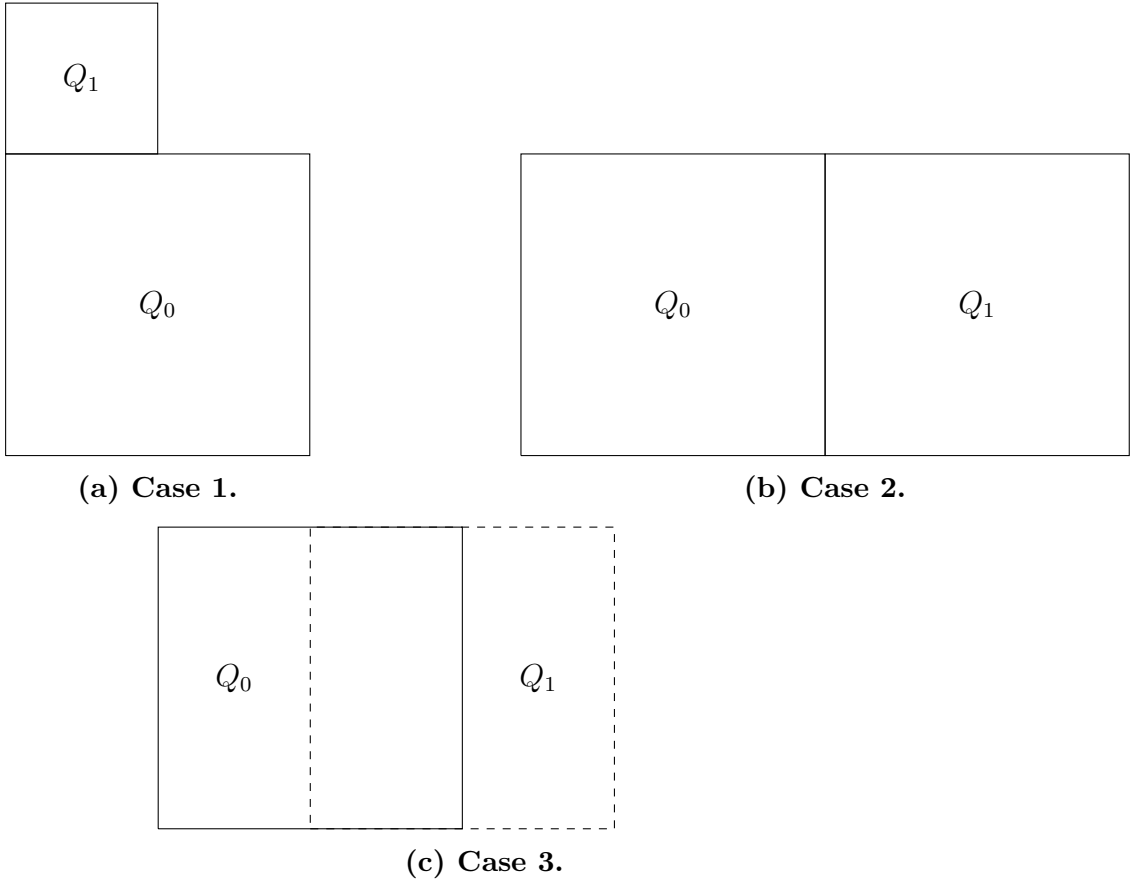

\begin{proof}

We prove \textit{Case 1}, the others being similar. For notational simplicity, we perform the following calculation when $a=0$, $L = I$ (i.e. $s_i=0$, $l_i=1$) and $\zeta_u = 0$ (which cannot be the case, but the calculation is the same as this amounts to an affine shift). We note that up to a shift in $w_0$ by $-R_\phi(\frac{1/2}{n}, 0)^T$, we may replace $\frac{i}{n}d$ by $\frac{i+1/2}{n}d$ in the definition of $z_i^\pm$ (\ref{def:znode}), which allows us to use midpoints of segments versus left end-points. This allows us to perform slightly cleaner estimates on $\phi$ and $w_0.$
 
We use an additional subscript to denote whether a quantity relates to $Q_0$ or $Q_1.$ We apply Lemma \ref{conti4.4} in $Q_0$ and $Q_1$ with grids $P_0$ and $P_1$, respectively, to find $w_{0,j}$ and $\phi_{j}$ for $j=0,1.$ It follows that 
\begin{equation}\label{bdd:uw1}
\|u_{0,0}^+ -w_{0,0}^+\|\leq C\sqrt{\sigma} d 
\end{equation} 
and 
\begin{equation} \label{bdd:uw2}
 \|u_{0,1}^-+u_{1,1}^--(w_{0,1}^-+w_{1,1}^-)\|\leq 2C\sqrt{\sigma}d.
\end{equation}
Furthermore, as $Q_0$ and $Q_1$ overlap at their top and bottom boundary respectively, we have 
\begin{equation}\label{eqn:uavgRel}
u_{0,0}^+ = \frac{1}{2}(u_{0,1}^-+u_{1,1}^-).
\end{equation} Consequently, using the definition of $w_{i,j}^\pm$, equation (\ref{eqn:uavgRel}), the triangle inequality, followed by application of the bounds (\ref{bdd:uw1}) and (\ref{bdd:uw2}), we find $$\|w_{0,0}-w_{0,1}+R_{\phi_0-\phi_1}((1/2)d/n,d)^T\| = \|w_{0,0}^+ - \frac{1}{2}(w_{0,1}^-+w_{1,1}^-) \|\leq C\sqrt{\sigma}d.$$ By a similar argument, since $u_{1,0}^+ = \frac{1}{2}(u_{2,1}^-+u_{3,1}^-)$, we find $$\|w_{0,0}-w_{0,1}+R_{\phi_0-\phi_1}((3/2)d/n,d)^T\|\leq C\sqrt{\sigma}d .$$ We note that to obtain both of these estimates is where we needed $n=4.$ Taking the difference of the terms, we find $$ (d/n)|\phi_0-\phi_1| = \|R_{\phi_0-\phi_1}(d/n,0)^T\|\leq C\sqrt{\sigma}d,$$ which implies $|\phi_0-\phi_1|\leq C\sqrt{\sigma}.$ From this, it also follows that $\|w_{0,0}-w_{0,1}\|\leq C\sqrt{\sigma}d.$

\end{proof}

With this in hand, we have enough tools to prove Theorem \ref{H1/2bound}.

\begin{proof}[Proof of Theorem \ref{H1/2bound}]
Given that the energy bounds of Lemma \ref{conti4.4} and equation (\ref{elasticgridenergy}) are independent of $c$, we do not concern ourselves with the function. We assume that $\zeta_u = \mu_0 e_0$. Shifting $u$ by the affine function $- \mu_0 e_0(x,y)^T,$ we can assume that one well is $\zeta_u = 0$ and the other well is $e_0.$

Fix the grid parameter $n = 4$. Let $\cup_k G^4_k$ be the grid as constructed in Theorem \ref{gridenergythm} with parameter $\delta>0$ for some $\bar y\in E$.
We write $$G^4_k = \bigcup_{i=1}^{i_{end}} P_{i,k}$$ where each parallelogram grid element $P_{i,k}$ is a translation of $L_k(z_k,z_{k+1})G^4$ and $P_{i_{end},k}$ is the rightmost grid element.
Choosing $\delta$  sufficiently small, each $P_{i,k}$ may be written as a translation of $(1+O(\delta))L(0,d_k)^2,$ with $|s|+|l-1| = O(\delta).$ Thus the results of Lemma \ref{conti4.4} still apply, and we find an associated pair $(w_{i,k},\phi_{i,k})$ satisfying the estimates of the lemma on the slightly rescaled grid $P_{i,k}.$

We now work to define our function $\bar u_y.$ For each $P_{i,k},$ we let $\gamma_{i,k}$ be the bottom left segment of the grid (in Lemma \ref{conti4.4} this would be on the interval $(0,d/n)\times \{0\}$). We denote the line average associated to this segment by 
\begin{equation} \label{def:linAvgSeg}
u_{i,k}:=\intbar_{\gamma_{i,k}} u \ d\mathcal{H}^1.
\end{equation} Note, for the last index $i_{end}$ for a fixed level $k$, we define $u_{i_{end}+1,k}$ to be the line average over the bottom right segment for the rightmost grid element $P_{i_{end},k}$.

For each $i,k,$ we let $z_{i,k}$ be the bottom left vertex of $P_{i,k}$ ($z_{i_{end}+1,k}$ being the bottom right of the rightmost grid element). As such, we may divide $P_{i,k}$ into two parallelograms $P_{i,k}^-$ and $P_{i,k}^+$, which each have a base of length $d_k/2 = d_{k+1},$ and have the vertex $z_{2i+1,k+1}$ in common. 

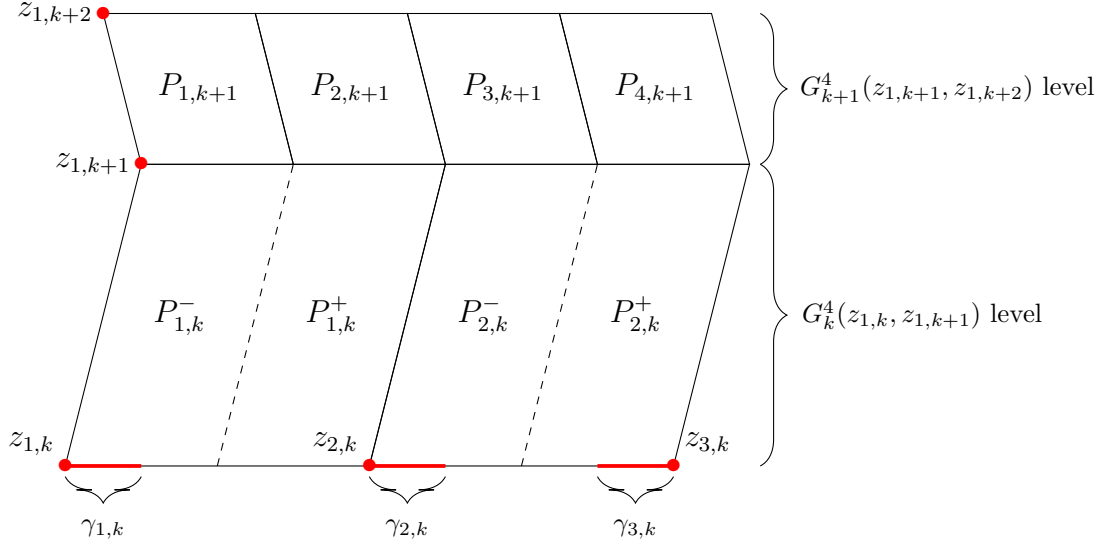
\begin{figure}[H]
\begin{center}
\begin{tikzpicture}
\draw (0,0) node[anchor= south east] {$z_{1,k}$} -- (1,4) -- (5,4) -- (4,0) -- (0,0);
\node [red] at (0,0) {\textbullet};
\node at (1.5,2) {$P_{1,k}^-$};
\node at (3.5,2) {$P_{1,k}^+$};
\draw [dashed] (2,0) -- (3,4);
\draw [decorate,decoration={brace,amplitude=10pt,mirror},xshift=0pt,yshift=-4pt]
(0,0) -- (1,0) node [black,midway,xshift=0.0cm,yshift=-0.7cm] 
{\footnotesize $\gamma_{1,k}$};
\draw [line width=0.5mm, red ] (0,0) -- (1,0);
\draw (4,0) node[anchor= south east] {$z_{2,k}$} -- (5,4) -- (9,4) -- (8,0) node[anchor= south west] {$z_{3,k}$} -- (4,0);
\node [red] at (4,0) {\textbullet};
\node [red] at (8,0) {\textbullet};
\node at (5.5,2) {$P_{2,k}^-$};
\node at (7.5,2) {$P_{2,k}^+$};
\draw [dashed] (6,0) -- (7,4);

\draw [decorate,decoration={brace,amplitude=10pt,mirror},xshift=0pt,yshift=-4pt]
(4,0) -- (5,0) node [black,midway,xshift=0.0cm,yshift=-0.7cm] 
{\footnotesize $\gamma_{2,k}$};
\draw [line width=0.5mm, red ] (4,0) -- (5,0);

\draw [decorate,decoration={brace,amplitude=10pt,mirror},xshift=0pt,yshift=-4pt]
(7,0) -- (8,0) node [black,midway,xshift=0.0cm,yshift=-0.7cm] 
{\footnotesize $\gamma_{3,k}$};
\draw [line width=0.5mm, red ] (7,0) -- (8,0);

\draw [decorate,decoration={brace,amplitude=10pt,mirror},xshift=4pt,yshift=0pt]
(9,0) -- (9,4) node [black,midway,xshift=2.13cm] 
{\footnotesize $G^4_k(z_{1,k},z_{1,k+1})$ level};
\draw (1,4) node[anchor= east] {$z_{1,k+1}$} -- (3,4) -- (2.5,6) -- (0.5,6) node[anchor= east] {$z_{1,k+2}$} -- (1,4);
\node [red] at (1,4) {\textbullet};
\node [red] at (0.5,6) {\textbullet};
\node at (1.75,5) {$P_{1,k+1}$};
\draw (3,4) -- (5,4) -- (4.5,6) -- (2.5,6) -- (3,4);
\node at (3.75,5) {$P_{2,k+1}$};
\draw (5,4) -- (7,4) -- (6.5,6) -- (4.5,6) -- (5,4);
\node at (5.75,5) {$P_{3,k+1}$};
\draw (7,4) -- (9,4) -- (8.5,6) -- (6.5,6) -- (7,4);
\node at (7.75,5) {$P_{4,k+1}$};
\draw [decorate,decoration={brace,amplitude=10pt,mirror},xshift=4pt,yshift=0pt]
(9,4) -- (9,6) node [black,midway,xshift=2.46cm] 
{\footnotesize $G^4_{k+1}(z_{1,k+1},z_{1,k+2})$ level};
\end{tikzpicture}
\end{center}
 \caption{\textbf{\small{Geometric quantities involved in the proof of Theorem \ref{H1/2bound}. }}}
 \end{figure}

We define $\bar u_y$ on $\text{conv} (P_{i,k})$ as follows: 
\begin{itemize}
\item Along the lower boundary, 
\begin{equation} \label{def:lowerBdryInterp}
\bar u_y(\theta z_{i,k}+(1-\theta)z_{i+1,k}) := \theta u_{i,k}+(1-\theta)u_{i+1,k},
\end{equation} for $\theta\in [0,1].$

\item Along the upper boundary, 
\begin{equation}\label{def:upperBdryInterp}
\bar u_y(\theta z_{2i+l,k+1}+(1-\theta)z_{2i+l+1,k+1}) := \theta u_{2i+l,k+1}+(1-\theta)u_{2i+l+1,k+1},
\end{equation} for $\theta\in [0,1], l=0,1,$ where $l$ designates whether we are considering the first (left) or second (right) half of the upper boundary.

\item Throughout the convex hull of $P_{i,k}$, 
\begin{equation}\label{def:hullInterp}
\bar u_y(\theta z+(1-\theta)(z+(z_{2i,k+1}-z_{i,k}))) := \theta \bar u_y( z)+(1-\theta)\bar u_y(z+(z_{2i,k+1}-z_{i,k})),
\end{equation} for all $z$ on the lower boundary of $P_{i,k},$ $\theta\in [0,1].$
\end{itemize}
In words, we define $\bar u_y$ on the vertices of $\text{conv} (P_{i,k})$ in terms of the associated averages of $u.$ Then we use linear interpolation to define the values on the upper and lower boundaries of $\text{conv} (P_{i,k}).$ Lastly, we interpolate between the lower and upper boundaries by moving in lines parallel to the sides of $\text{conv} (P_{i,k}).$

Given this construction of $\bar u_y,$ we now wish to show that in each parallelogram $\text{conv}(P_{i,k})$, $\nabla \bar u_y$ is close to the skew symmetric matrix $R_{\phi_{i,k}}$. We restrict our attention to grid elements which are not the rightmost, a simpler case. We introduce the parallelogram grid $P'_{i,k} = P_{i,k}^+\cup P_{i+1,k}^-$ for which Lemma \ref{conti4.4} applies (associated terms have apostrophe, i.e. $\phi_{i,k}'$). Define $$\nu_1 := (1,0) = \frac{z_{i+1,k} - z_{i,k}}{\|z_{i+1,k} - z_{i,k}\|}, \quad \nu_2:= \frac{z_{2i,k+1} - z_{i,k}}{\|z_{2i,k+1} - z_{i,k}\|} .$$
As $\nu_1$ and $\nu_2$ are linearly independent, we have

\begin{equation} \label{bdd:gradnu1nu2}
\begin{aligned}
\|\nabla \bar u_y - R_{\phi_{i,k}}\|_{L^\infty (\text{conv}(P_{i,k}))} \leq & \left\|\frac{\partial}{\partial \nu_1} \bar u_y - R_{\phi_{i,k}}(\nu_1)\right\|_{L^\infty (\text{conv}(P_{i,k}))} \\
& +\left\|\frac{\partial}{\partial \nu_2} \bar u_y- R_{\phi_{i,k}}(\nu_2)\right\|_{L^\infty (\text{conv}(P_{i,k}))}.
\end{aligned}
\end{equation}
As $\bar u_y$ is constructed via linear interpolations (\ref{def:lowerBdryInterp}), (\ref{def:upperBdryInterp}), (\ref{def:hullInterp}), we bound $\|\frac{\partial}{\partial \nu_1} \bar u_y - R_{\phi_{i,k}}(\nu_1)\|_{L^\infty (\text{conv}(P_{i,k}^-))}$ via difference quotients along the top and bottom boundary of $P_{i,k}^-$:
\begin{equation} \label{eqn:nu1Gradient}
\begin{aligned}
\left\|\frac{\partial}{\partial \nu_1} \bar u_y - R_{\phi_{i,k}}(\nu_1)\right\|_{L^\infty (\text{conv}(P_{i,k}^-))}  \leq &  C \left( \left\| \frac{(\bar u_y - R_{\phi_{i,k}})(z_{i,k}') - (\bar u_y - R_{\phi_{i,k}})(z_{i,k})}{\|z_{i,k}' - z_{i,k}\|} \right\| \right.  \\
& \left. + \left\| \frac{(\bar u_y - R_{\phi_{i,k}})(z_{2i+1,k+1}) - (\bar u_y - R_{\phi_{i,k}})(z_{2i,k+1})}{\|z_{2i+1,k+1} - z_{2i,k+1}\|} \right\| \right) \\
= & C\left( \left\| \frac{u_{i+1,k} - u_{i,k}}{d_k} -R_{\phi_{i,k}}(1,0)^T \right\| \right. \\
& + \left. \left\| \frac{u_{2i+1,k+1} - u_{2i,k+1}}{d_{k+1}} -R_{\phi_{i,k}}(1,0)^T \right\| \right).
\end{aligned}
\end{equation}
Similarly, 
\begin{equation}\label{eqn:nu2Gradient}
\begin{aligned}
\biggl\| \frac{\partial}{\partial \nu_2} \bar u_y & - R_{\phi_{i,k}}(\nu_2) \biggl\|_{L^\infty (\text{conv}(P_{i,k}^-))} & \\
\leq &  C\left(  \left\|\frac{u_{2i,k+1}-u_{i,k}}{\|z_{2i,k+1}-z_{i,k}\|}- R_{\phi_{i,k}} \left(\frac{z_{2i,k+1}-z_{i,k}}{\|z_{2i,k+1}-z_{i,k}\|}\right)\right\| \right. \\
& + \left. \left\|\frac{u_{2i+1,k+1}-\frac{1}{2}(u_{i,k}+u_{i+1,k})}{\|z_{2i+1,k+1}-\frac{1}{2}(z_{i,k}+z_{i+1,k})\|}  -R_{\phi_{i,k}} \left(\frac{z_{2i+1,k+1}-\frac{1}{2}(z_{i,k}+z_{i+1,k})}{\|z_{2i+1,k+1}-\frac{1}{2}(z_{i,k}+z_{i+1,k})\|}\right)\right\|  \right) 
\end{aligned}
\end{equation}
The bounds over $P_{i,k}^+$ are once again similar and we do not state them.

We bound the horizontal finite difference along the lower boundary of $P_{i,k},$ which will account for both terms on the right hand side of (\ref{eqn:nu1Gradient}) up to an application of Lemma \ref{lem:adjacentBoxes}. Define 
$$\sigma'_{i,k} : = \intbar_{\text{conv}(P_{i,k}\cup P_{i+1,k})}\min\{\|e(u)\|^2,\|e(u)-e_0\|^2\} \ dz +  \frac{1}{d_k}\int_{P_{i,k}\cup P_{i+1,k}} \|e(u)\|^2\ d\mathcal{H}^1 ,$$ where the integral is performed over $P_{i+1,k}$ versus $P_{i,k}'$ for convenience, not necessity.
We compute
\begin{equation}\label{bdd:FD1}
\begin{aligned}
\biggl\| & \frac{u_{i+1,k}  - u_{i,k}}{d_k} -R_{\phi_{i,k}}(1,0)^T \biggl\|^2 \\
 \leq &  \frac{1}{d_k^2}\Big(\| u_{i+1,k} - u_{i,k}' - R_{\phi_{i,k}}(d_k/2,0)^T\|^2  + \| u_{i,k}' - u_{i,k} - R_{\phi_{i,k}}(d_k/2,0)^T\|^2\Big) \\
\leq & \frac{C}{d_k^2}\Big(\| u_{i+1,k} - u_{i,k}' - R_{\phi_{i,k}'}(d_k/2,0)^T\|^2 + \|R_{\phi_{i,k}'-\phi_{i,k}}(d_k/2,0)^T\|^2  \\
& +\| u_{i,k}'-w_{i,k}- R_{\phi_{i,k}}(z_{i,k}')\|^2 +\|u_{i,k} -w_{i,k}  - R_{\phi_{i,k}}(z_{i,k})\|^2\Big) \\
\leq &  \frac{C}{d_k^2}\Big(\| u_{i+1,k}-w_{i,k}'- R_{\phi_{i,k}'}(z_{i+1,k})\|^2  +\|u_{i,k}' -w_{i,k}'  - R_{\phi_{i,k}'}(z_{i,k}')\|^2 \\
&  + d_k^2|\phi_{i,k}'-\phi_{i,k}|^2 +C\sigma_{i,k}'d_k^2\Big) \\
\leq & C (\sigma_{i,k}' + |\phi_{i,k}'-\phi_{i,k}|^2) \leq  C \sigma_{i,k}',
\end{aligned}
\end{equation}
where we have used that $$z_{i,k}'-z_{i,k}=z_{i+1,k}-z_{i,k}' =  (d_k/2,0)$$ and
\begin{equation}\label{phihoriz}
|\phi_{i,k}'-\phi_{i,k}|^2\leq C\sigma_{i,k}',
\end{equation}
by Lemma \ref{lem:adjacentBoxes},
and 
\begin{equation} \nonumber
\| u_{i,k}'-w_{i,k}- R_{\phi_{i,k}}(z_{i,k}')\| \leq C\sqrt{\sigma_{i,k}'}d_k
\end{equation}
 along with 
 \begin{equation}\nonumber
\|u_{i,k}' -w_{i,k}'  - R_{\phi_{i,k}'}(z_{i,k}')\| \leq C\sqrt{\sigma_{i,k}'}d_k,
\end{equation}
which are consequences of Lemma \ref{conti4.4} with $\zeta_u = 0$ applied to $P_{i,k}$ and $P_{i,k}'$, respectively (note that in the notation of Lemma \ref{conti4.4}, $u_{i,k}'$ is $u_{2}^-$ associated with the grid $P_{i,k}$).

 We define 
 \begin{equation} \nonumber
 \begin{aligned}
 \sigma_{i,k}: = & \intbar_{\text{conv}(P_{i,k})}\min\{\|e(u)\|^2,\|e(u)-e_0\|^2\} +  \frac{1}{d_k}\int_{P_{i,k}} \|e(u)\|^2\ d\mathcal{H}^1 \\
 &+ \intbar_{\text{conv}(P_{2i,k+1})}\min\{\|e(u)\|^2,\|e(u)- e_0\|^2\} +  \frac{1}{d_{k+1}}\int_{P_{2i,k+1}} \|e(u)\|^2\ d\mathcal{H}^1 \\
 & + \intbar_{\text{conv}(P_{2i+1,k+1})}\min\{\|e(u)\|^2,\|e(u)- e_0\|^2\} +  \frac{1}{d_{k+1}}\int_{P_{2i+1,k+1}} \|e(u)\|^2\ d\mathcal{H}^1.
 \end{aligned}
\end{equation}
We note that $\|z_{2i,k+1}-z_{i,k}\| = (1+O(\delta))d_k$ by construction. Furthermore, up to translation, we have that $z_{i,k} = 0$, and $\|z_{2i,k+1}\| = (1+O(\delta))d_k$.  Using Lemma \ref{conti4.4} and Lemma \ref{lem:adjacentBoxes}, we compute a finite difference in the direction of $\nu_2 = \frac{z_{2i,k+1}-z_{i,k}}{\|z_{2i,k+1}-z_{i,k}\|}$ on the left boundary of $\text{conv} (P_{i,k})$. This estimate will be used to bound the first term of the right hand side of (\ref{eqn:nu2Gradient}).
 \begin{equation} \label{bdd:FD2} 
\begin{aligned}
 \biggl\| & \frac{u_{2i,k+1}-u_{i,k}}{\|z_{2i,k+1}-z_{i,k}\|}- R_{\phi_{i,k}} \left(\frac{z_{2i,k+1}-z_{i,k}}{\|z_{2i,k+1}-z_{i,k}\|}\right)\biggl\|^2 \\ 
 \leq  & \frac{C}{\|z_{2i,k+1}-z_{i,k}\|^2}\Big( \| u_{2i,k+1} - (w_{2i,k+1}+R_{\phi_{2i,k+1}}(z_{2i,k+1}))\|^2 + \| u_{i,k} - (w_{i,k}+R_{\phi_{i,k}}(z_{i,k}))\|^2 \\
 &  + \|w_{2i,k+1}-w_{i,k}\|^2 +\|R_{\phi_{2i,k+1}-\phi_{i,k}}(z_{2i,k+1})\|^2 \Big)  \leq  C\sigma_{i,k}.
\end{aligned}
\end{equation}
Note that the integrals in the definition of $\sigma_{i,k}$ associated with $P_{2i+1,k+1}$ are not needed for the above inequality, but will be necessary for the next bound. 

We perform a similar calculation for near vertical finite differences along the common boundary of $\text{conv} (P_{i,k}^-)$ and $\text{conv} (P_{i,k}^+)$. This estimate bounds the second term of the right hand side of (\ref{eqn:nu2Gradient}). Using that $z_{2i+1,k+1} = \frac{1}{2}(z_{2i,k+1}+z_{2i+2,k+1})$ and adding and subtracting the term $\frac{1}{2}(u_{2i,k+1}+u_{2i+2,k+1})$, we estimate
%
\begin{equation} \label{bdd:FD3}
\begin{aligned}
\biggl\| & \frac{u_{2i+1,k+1}-\frac{1}{2}(u_{i,k}+u_{i+1,k})}{\|z_{2i+1,k+1}-\frac{1}{2}(z_{i,k}+z_{i+1,k})\|}  -R_{\phi_{i,k}} \left(\frac{z_{2i+1,k+1}-\frac{1}{2}(z_{i,k}+z_{i+1,k})}{\|z_{2i+1,k+1}-\frac{1}{2}(z_{i,k}+z_{i+1,k})\|}\right)\biggl\|^2 \\
\leq & \frac{C}{d_k^2}\biggl(\frac{1}{4} \|u_{2i,k+1}-u_{i,k}-R_{\phi_{i,k}} (z_{2i,k+1}-z_{i,k})\|^2 + \frac{1}{4}\|u_{2i+2,k+1}-u_{i+1,k}-R_{\phi_{i,k}} (z_{2i+2,k+1}-z_{i+1,k})\|^2 \\
& + \left\|\frac{1}{2}(u_{2i+2,k+1} + u_{2i,k+1}) - u_{2i+1,k+1}\right\|^2\biggl) \\
\leq & C\biggl(\sigma_{i,k}+\sigma_{i+1,k}  + \frac{1}{d_k^2}\biggl\|\frac{1}{2}(u_{2i+2,k+1} + u_{2i,k+1}) - u_{2i+1,k+1}\biggl\|^2\biggl) \\
\leq & C (\sigma_{i,k} +\sigma_{i+1,k} +\sigma'_{2i,k+1}+\sigma'_{2i+2,k+1}),
\end{aligned}
\end{equation}
where in the second inequality we have applied the analysis of finite differences along the left boundaries and the bound $|\phi_{i+1,k}-\phi_{i,k}|^2\leq C\sigma_{i,k}'$ provided by Lemma \ref{lem:adjacentBoxes}. To see the last inequality, we note 
\begin{align*}
\biggl\|& \frac{1}{2}(u_{2i+2,k+1} + u_{2i,k+1}) - u_{2i+1,k+1}\biggl\| \\
 \leq & \|u_{2i+2,k+1} - u_{2i+1,k+1} -R_{\phi_{2i+1,k+1}}(d_{k+1},0)^T\| + \|u_{2i+1,k+1} - u_{2i,k+1} -R_{\phi_{2i+1,k+1}}(d_{k+1},0)^T\|,
\end{align*}
which are the horizontal finite differences, modulo a term like (\ref{phihoriz}) for the second term, which have already been analyzed, thus concluding the bound.

We define $$R_\phi := \sum_{i,k}\chi_{\text{conv}(P_{i,k})}R_{\phi_{i,k}},$$ noting that by (\ref{def:skewsymRot}), $(R_\phi)^{\rm{sym}} = 0$ almost everywhere. Let $G := \bigcup_k \text{conv}(G_k^n).$ Applying (\ref{bdd:gradnu1nu2}), (\ref{eqn:nu1Gradient}), (\ref{eqn:nu2Gradient}), and the subsequent finite difference estimates (\ref{bdd:FD1}), (\ref{bdd:FD2}), (\ref{bdd:FD3}), we have
\begin{align*}
\|\nabla \bar u_y - & R_\phi \|_{L^2(G)}^2 \\
\leq & C \sum_{i,k}\mathcal{L}^2(\text{conv}(P_{i,k}))\biggl( \biggl\| \frac{u_{i+1,k} - u_{i,k}}{d_k} -R_{\phi_{i,k}}(1,0)^T  \biggl\|^2 \\
& + \biggl\|\frac{u_{2i,k+1}-u_{i,k}}{\|z_{2i,k+1}-z_{i,k}\|}-R_{\phi_{i,k}} \left(\frac{z_{2i,k+1}-z_{i,k}}{\|z_{2i,k+1}-z_{i,k}\|}\right)\biggl\|^2 \\
& + \biggl\|\frac{u_{2i+1,k+1}-\frac{1}{2}(u_{i,k}+u_{i+1,k})}{\|z_{2i,k+1}-z_{i,k}\|}-R_{\phi_{i,k}} \left(\frac{z_{2i+1,k+1}-\frac{1}{2}(z_{i,k}+z_{i+1,k})}{\|z_{2i+1,k+1}-\frac{1}{2}(z_{i,k}+z_{i+1,k})\|}\right)\biggl\|^2\biggl) \\
\leq & C \sum_{i,k}\mathcal{L}^2(\text{conv}(P_{i,k})) ( \sigma_{i,k} +\sigma_{i+1,k} +\sigma'_{2i,k+1}+\sigma'_{2i+2,k+1} ) \\
\leq & C \sum_{i,k} \mathcal{L}^2(\text{conv}(P_{i,k})) \biggl(\intbar_{\text{conv}(P_{i,k})}\min\{\|e(u)\|,\|e(u)-e_0\|\}^2 +  \frac{1}{d_k}\int_{P_{i,k}} \|e(u)\|^2\ d\mathcal{H}^1\biggl) \\
\leq & \int_{G}\min\{\|e(u)\|^2,\|e(u)-e_0\|^2\} \ dz+ \sum_{k} d_k \sum_{i}\int_{P_{i,k}} \|e(u)\|^2\ d\mathcal{H}^1 \\
\leq & C\eta \epsilon +  (\sum_k d_k) \int_{G_k^n} \|e(u)\|^2\ d\mathcal{H}^1\leq  C\eta \epsilon,
\end{align*}
where in the last line we have applied the energy bounds from Theorem \ref{gridenergythm}, and the bound in the second to last line follows by undoing the affine shift of $u$ and using $I_\epsilon[u,c,(-d,d)\times (l_0,l_1)]\leq \eta$ in conjunction with $f$ being a super-quadratic well. As $(-d/2,d/2)\times (y-C_{d,l},y)\subset G,$ we have 
\begin{align*}
\|e(\bar u_y) \|_{L^2((-d/2,d/2)\times (y-1/2,y))}^2 \leq &C\|\nabla \bar u_y - R_\phi \|_{L^2(G)}^2 \leq C\eta \epsilon.
\end{align*}

Applying Korn's inequality (see \cite{Nitsche1981-Korn}), subsequently the trace theorem (see \cite{LeoniBook}), and noting by continuity that $\bar u_y(\cdot, y) = u(\cdot,y),$ we conclude the proof.

\end{proof}

\begin{proof}[Proof of Theorem \ref{boxenergy}]
We construct the desired sequence by forming a transitional layer of thickness $\epsilon_i$ on the upper and lower halves of the box. We treat the upper half; the lower half is analogous. Let $\psi:\R\to [0,1]$ be a smooth cutoff function with $\psi(x) =1$ for $x<0$ and $\psi(x)=0$ for $x>1$. For some $y_i\in (l/2,3l/4)$ to be determined, let $\psi_i(x,y) := \psi((y-y_i)/\epsilon_i)$. We define 
\begin{align}\label{def:cTransFunc}
& \bar c_i := \psi_i c_i+(1-\psi_i)\bar c.
\end{align} 
We must be more cautious in defining $\bar u_i$ as previously noted.

By Proposition \ref{FonsecaTartarProp}, 
$$\liminf\limits_{i\to \infty} I_{\epsilon_i}[u_i,c_i,(-2d,2d)\times(-l/8,l/8)]\geq 4d \k(e_y), $$ and therefore by (\ref{hyp:boxenergy}), 
$$\lim\limits_{i \to \infty} I_{\epsilon_i}[u_i,c_i,(-2d,2d)\times(l/4,l)] = 0. $$ For computational simplicity, we perturb the hypotheses of the theorem to consider 
\begin{equation} \label{eqn:uSkwShift}
u_i\to \bar u =:\bar u_{e_y}(x,y) - S_{e_y}(x,y)^T \quad \text{ in } H^1(D_{2d},\R^2)
\end{equation} and $$c_i\to \bar c =: \bar c_{e_y} \quad \text{ in }L^2(D_{2d})$$ (see (\ref{unucnu}) and (\ref{def:DBox}) for relevant definitions).
Hence 
\begin{equation}\label{def:eta1}
\eta_i := \|c_i-\bar c\|^2_{L^2}+\|u_i-\bar u\|^2_{H^1}+\mathcal{L}^2(\{|c_i-\bar c|\geq 1/2-\mu_0\}) +  I_{{\epsilon_i}}[u_i,c_i,(-2d,2d)\times(l/4,l)]\to 0.
\end{equation}  By Theorem \ref{H1/2bound} for each $i$ sufficiently large, there is a set $E_i\subset (l/2,3l/4)$ such that $\mathcal{L}^1(E_i)> l/8$ and for all $y_0\in E_i$ there is an affine function 
\begin{equation}\label{def:affineFuncApprox}
w_{y_0}(x,y) :=(\mu_1e_0+ R_{\phi_{y_0}})(x,y)^T+a_{y_0}
\end{equation} (depending on $i$) such that 
\begin{equation}\label{bdd:affineFuncApprox}
\|u_i - w_{y_0}\|_{H^{1/2}((-d,d)\times\{y_0\})}^2\leq C\eta_i \epsilon_i.
\end{equation}

Modifying a proof of Gagliardo's (see Lemma \ref{traceext} below this proof), we may construct $v_{y_0} \in H^{1}((d/2,d/2)\times (y_0,l),\R^2)$ satisfying 
\begin{equation}\label{gagres}
\begin{aligned}
& v_{y_0} =  u_i - w_{y_0} \ \ \text{ on } (-d/2,d/2)\times\{y_0\}\\
& v_{y_0} =  0 \ \ \text{  on some neighborhood of }\{(x,y): y = l \} \\
& \| v_{y_0}\|_{H^1({(d/2,d/2)\times (y_0,l)})}^2 \leq  C \eta_i \epsilon_i. 
\end{aligned}
\end{equation}
Define 
\begin{equation}\label{def:uTransFunc}
\bar u_i := \psi_i u_i +(1-\psi_i)(v_{i} +w_{i}),
\end{equation}
where $v_i = v_{y_i},$ $w_i = w_{y_i}$, and $y_i\in E_i$ is to be determined.
We compute the energy for the constructed sequence (recall (\ref{def:DBox})): 
\begin{align*}
 I_{{\epsilon_i}}[\bar u_i,\bar c_i,D_{d/2}] = & I_{{\epsilon_i}}[ u_i, c_i,D^-_{d/2,\epsilon_i}] +\int_{D_{d/2,\epsilon_i}}\frac{1}{\epsilon_i}f(\bar c_i) \ dz +\int_{D_{d/2,\epsilon_i}}\epsilon_i\|\nabla \bar c_i\|^2 \ dz  \\
 &+\int_{D_{d/2,\epsilon_i}}\frac{1}{\epsilon_i}\mathbb{C}(e(\bar u_i)-\bar c_i e_0):(e(\bar u_i)-\bar c_i e_0) \ dz + \int_{D_{d/2,\epsilon_i}^+}\frac{1}{\epsilon_i}\mathbb{C}(e(v_i)):e(v_i) \ dz \\
 = : &  A_1+A_2+A_3+A_4+A_5.
\end{align*}

We will bound terms $A_2, \ A_3, \ A_4,$ and $A_5$ by $\eta_i$ for appropriate choices of $y_i$ and explicitly compute the limit of energy $A_1$.

\noindent \textbf{Term $\boldsymbol{A_2}$:} By (\ref{def:cTransFunc}),
\begin{equation} \label{eqn:A2parts}
\begin{aligned}
A_2 = & \int_{D_{d/2,\epsilon_i}}\frac{1}{\epsilon_i}f(\psi_i c_i+(1-\psi_i)\bar c) \ dz \\
= & \int_{D_{d/2,\epsilon_i}\cap \{|c_i-\bar c|< 1/2-\mu_0\}}\frac{1}{\epsilon_i}f(\psi_i c_i+(1-\psi_i)\bar c) \ dz \\
& + \int_{D_{d/2,\epsilon_i}\cap \{|c_i-\bar c|\geq 1/2 -\mu_0\}}\frac{1}{\epsilon_i}f(\psi_i c_i+(1-\psi_i)\bar c) \ dz \\
=:& A_{21} + A_{22}.
\end{aligned}
\end{equation}
To bound $A_{22}$, we integrate $y_i$ over $(l/2,3l/4)$ and apply Fubini's Theorem to find 
\begin{equation} \label{bdd:FubiniArg}
\begin{aligned}
\int_{l/2}^{3l/4}\frac{1}{\epsilon_i}\int_{D_{d/2,\epsilon_i}}\chi_{\{|c_i-\bar c|\geq 1/2-\mu_0\}} & (x,y)\ d(x,y) \ dy_i  \\
=&
\int_{l/2}^{3l/4}\frac{1}{\epsilon_i}\int_{y_i}^{y_i+\epsilon_i}\int_{-d/2}^{d/2}\chi_{\{|c_i-\bar c|\geq 1/2-\mu_0\}}(x,y)\ dx \ dy \ dy_i  \\
=& \frac{1}{\epsilon_i}\int_{0}^{\epsilon_i}\int_{l/2}^{3l/4}\int_{-d/2}^{d/2}\chi_{\{|c_i-\bar c|\geq 1/2-\mu_0\}}(x,y_i+y)\ dx \ dy \ dy_i \\
\leq & \int_{l/4}^l \int_{-d/2}^{d/2} \chi_{\{|c_i - \bar c|\geq 1/2-\mu_0\}}(x,t) \ dx \ dt\leq \eta_i.
\end{aligned}
\end{equation}
By Lemma \ref{energylemma}, for $\theta\in (0,1)$ there exists $E_{1,\theta}\subset (l/2,3l/4)$ with $\mathcal{L}^1(E_{1,\theta})>\theta l/4$ such that $$ \frac{1}{\epsilon_i}\int_{D_{d/2,\epsilon_i}} \chi_{\{|c_i-\bar c|\geq 1/2-\mu_0\}}(x,y) \ d(x,y) \leq C_\theta \eta_i$$
for all $y_i\in E_{1,\theta}.$
Hence 
\begin{equation} \label{bdd:A22}
A_{22}\leq C_\theta \|f\|_{\infty}\eta_i.
\end{equation}
 To estimate $A_{21}$, we use that $f$ is decreasing on the interval $[1/2,\mu_1]$ and increasing on $[\mu_1,1]$ (see Proposition \ref{prop:ffunc}), and that in $D_{2d,\epsilon_i}$,  we have $\bar c = \mu_1.$ Supposing $c_i \in [1/2,\mu_1]$, we find  $\psi_i c_i+(1-\psi_i)\bar c\geq c_i\geq 1/2$, and consequently $f(\psi_i c_i+(1-\psi_i)\bar c)\leq f(c_i)$, implying 
 \begin{equation} \label{bdd:A21}
 A_{21} \leq \int_{D_{d/2,\epsilon}} \frac{1}{\epsilon_i} f(c_i) \ dz.
 \end{equation}
Combining (\ref{bdd:A22}) and (\ref{bdd:A21}), we have 
\begin{equation}\label{bdd:A2}
A_2\leq C_\theta \eta_i.
\end{equation}
\noindent \textbf{Term $\boldsymbol{A_3}$:}
 By (\ref{def:cTransFunc}), we have
\begin{align*}
\int_{\tilde \Omega_{\epsilon_i}}\epsilon_i\|\nabla \bar c_i\|^2 \ dz =& \int_{\tilde \Omega_{\epsilon_i}}\epsilon_i\|\psi_i \nabla c_i+(c_i-\bar c)\nabla \psi_i\|^2 \ dz \\
\leq & C \int_{\tilde \Omega_{\epsilon_i}}\epsilon_i\|\nabla c_i\|^2 \ dz + C\|\nabla \psi\|_\infty\frac{1}{\epsilon_i}\int_{\tilde \Omega_{\epsilon_i}} |c_i-\bar c|^2 \ dz.
\end{align*}
As in (\ref{bdd:FubiniArg}), by integrating in $y_i$ over $(l/2,3l/4)$ and applying Fubini's Theorem and a change of variables,
\begin{equation} \label{bdd:FubiniArg2}
\begin{aligned}
\int_{l/2}^{3l/4}\frac{1}{\epsilon_i}\int_{D_{d/2,\epsilon_i}}|c_i(x,y) - \bar c(x,y)| & \ d(x,y) \ dy_i  \\
=&
\int_{l/2}^{3l/4}\frac{1}{\epsilon_i}\int_{y_i}^{y_i+\epsilon_i}\int_{-d/2}^{d/2}|c_i(x,y) - \bar c(x,y)|\ dx \ dy \ dy_i  \\
=& \frac{1}{\epsilon_i}\int_{0}^{\epsilon_i}\int_{l/2}^{3l/4}\int_{-d/2}^{d/2}|c_i(x,y_i+y) - \bar c(x,y_i+y)|\ dx \ dy \ dy_i \\
\leq & \int_{l/4}^l \int_{-d/2}^{d/2} |c_i(x,t) - \bar c(x,t)| \ dx \ dt\leq C\eta_i.
\end{aligned}
\end{equation}
By Lemma \ref{energylemma}, for $\theta\in (0,1)$ there exists $E_{2,\theta}\subset (l/2,3l/4)$ with $\mathcal{L}^1(E_{2,\theta})>\theta l/4$ such that $$ \frac{1}{\epsilon_i}\int_{D_{d/2,\epsilon_i}} |c_i-\bar c| \ dz \leq C_\theta \eta_i$$
for all $y_i\in E_{2,\theta}.$
Hence 
\begin{equation}\label{bdd:A3}
A_3 \leq C_\theta\eta_i.
\end{equation}

\noindent \textbf{Term $\boldsymbol{A_4}$:} We now estimate the elastic energy on the transition layer: By (\ref{def:cTransFunc}) and (\ref{def:uTransFunc}) we have
\begin{align}
A_4 \leq& \frac{C}{\epsilon_i}\int_{D_{d/2,\epsilon_i}}\|\psi_i (e(u_i) - c_ie_0) + (1-\psi_i)(e(v_i+w_{i})-\bar c e_0) + ((u_i-w_{i}-v_i)\otimes \nabla \psi_i)^{\rm{sym}}\|^2 \ dz \nonumber \\
\leq & \frac{C}{\epsilon_i}\int_{D_{d/2,\epsilon_i}} \Big(\|e(u_i) -c_ie_0\|^2 + \|\nabla v_i\|^2\Big) \ dz + \frac{C}{\epsilon_i^3}\int_{D_{d/2,\epsilon_i}}\|u_i-w_{i}-v_i\|^2 \ dz \nonumber \\
=: & A_{41} + A_{42}, \label{bdd:A4parts}
\end{align}
where we have used that in $D_{2d,\epsilon_i}$, $\bar c = \mu_1$ by definition (\ref{unucnu}) and that $e(w_{i}) = \mu_1e_0$ by (\ref{def:affineFuncApprox}).
By (\ref{coercivity}) and (\ref{gagres}), $A_{41}$ is controlled by $C\eta_i$. To bound $A_{42}$, we utilize the Poincar\'e inequality in $D_{d/2,\epsilon}$ as $u_i-w_{i}-v_i = 0$ on the lower boundary of this domain by (\ref{gagres}) (see proof of the Poincar\'e inequality in \cite{LeoniBook}). Explicitly,
\begin{equation}\label{bdd:A42}
\begin{aligned}
A_{42} \leq  & \frac{C}{\epsilon_i} \int_{D_{d/2,\epsilon_i}}\|\nabla(u_i-w_{i}-v)\|^2 \ dz \\
  \leq & \frac{C}{\epsilon_i} \int_{D_{d/2,\epsilon_i}}  \|\nabla u_i - \mu_1 e_0\|^2 + \|\phi_{y_i}\|^2 +\|\nabla v_i\|^2 \ dz,
\end{aligned} 
\end{equation}
where in the last inequality we have used (\ref{def:skewsymRot}) and (\ref{def:affineFuncApprox}).

Reasoning as in the proof of (\ref{bdd:FubiniArg}) and (\ref{bdd:FubiniArg2}), we may apply Lemma \ref{boxenergy} to find a set $E_{3,\theta}\subset (l/2,3l/4)$ with $\mathcal{L}^1(E_{3,\theta}) >\theta l/4$ such that 
\begin{equation}\label{bdd:graduEta}
\frac{C}{\epsilon_i} \int_{D_{d/2,\epsilon_i}}  \|\nabla u_i - \mu_1 e_0\|^2 \ dz \leq C_\theta\eta_i.
\end{equation}

The last term in the integrand on the right side of (\ref{bdd:A42}) is controlled by (\ref{gagres}). Thus, it remains to control $\phi_i:=\phi_{y_i}$ by $\eta_i$; to do this, we must first bound the constant $a_i:=a_{y_i}$ in (\ref{def:affineFuncApprox}). Applying Lemma \ref{energylemma} to $\|u_i-\bar u\|^{2}_{L^2(D_{d/2,\epsilon_i})}$, there is a set $E_{4,\theta}\subset (l/2,3l/4)$, with $\mathcal{L}^1(E_{4,\theta})>\theta l/4$, such that for all $y_i \in E_{4,\theta}\subset (l/2,3l/4),$
$$\int_{-d/2}^{d/2} \| u_i(x,y_i) - \mu_1 e_0(x,y_i)^T\|^2 \ d\mathcal{H}^1 \leq C_\theta\eta_i ,$$ where we have used (\ref{eqn:uSkwShift}). Consequently, supposing $y_i \in E_0\cap E_{4,\theta}$, we are able to compute
\begin{equation}\label{bdd:aconst}
\begin{aligned}
|a_i^{(2)}|^2 \leq & \left|\intbar_{-d/2}^{d/2}  u_i(x,y_i) - \mu_1 e_0(x,y_i)^T)^{(2)}  \ dx\right|^2 + \left|\intbar_{-d/2}^{d/2} u_i(x,y_i)-w_{i}(x,y_i))^{(2)} \ dx \right|^2 \\
\leq & C\eta_i + C\eta_i \epsilon_i \leq  C\eta_i.
\end{aligned}
\end{equation}
where we have used (\ref{def:skewsymRot}), (\ref{bdd:affineFuncApprox}), the fact that $\intbar_{-d/2}^{d/2} \phi_i x \ dx = 0$, and the notation $z= (z^{(1)},z^{(2)})$ for a vector $z\in \R^2$. With this in hand, we may estimate 
\begin{align*}
\frac{d^3}{12}\|\phi_i\|^2 = \int_{-d/2}^{d/2}\|\phi_i\|^2x^2 \ dx \leq & C\Big(|(a_i)^{(2)}|^2 \ dx + \int_{-d/2}^{d/2}| ( u_i(x,y_i) - \mu_1 e_0(x,y_i)^T)^{(2)} |^2 \ dx \\
& + \int_{-d/2}^{d/2}|( u_i-w_{i})^{(2)}|^2 \ dx\Big) \\
\leq & C\eta_i.
\end{align*}
By a similar argument, one can conclude $|(a_i)^{(1)}|^2\leq C\eta_i$ too. Combining (\ref{bdd:A42}), (\ref{bdd:graduEta}), (\ref{bdd:aconst}), and the previous inequalities, we conclude
$$ \frac{C}{\epsilon_i^3}\int_{D_{d/2,\epsilon_i}}\|u_i-w_{i}-v_i\|^2 \leq C\eta_i.$$ By (\ref{bdd:A4parts}) this implies
\begin{equation}\label{bdd:A4}
A_4\leq C_\theta \eta_i.
\end{equation}

\noindent \textbf{Term $\boldsymbol{A_5}$:} By construction of $v_i$ (see (\ref{gagres})), we have that 
\begin{equation}\label{bdd:A5}
A_5 = \int_{D_{d/2,\epsilon_i}^+}\frac{1}{\epsilon_i}\mathbb{C}(e(v_i)):e(v_i)\ dz \leq C \int_{D_{d/2,\epsilon_i}^+}\frac{1}{\epsilon_i}\|\nabla v_i\|^2 \ dz  \leq C\eta_i.
\end{equation}
\noindent \textbf{Term $\boldsymbol{A_1}$:} We may apply Proposition \ref{FonsecaTartarProp} and Theorem \ref{thm:liminf} to see 
\begin{align*}
\liminf_{i\to \infty}  I_{{\epsilon_i}}[ u_i, c_i,D^-_{d/2,\epsilon_i}] \geq & \liminf_{i\to \infty}  I_{{\epsilon_i}}[ u_i, c_i,(-d/2,d/2)\times(-l,l/4)] \\
\geq & d\k(e_y).
\end{align*}
The upper bound follows by contradiction. Suppose that $\limsup\limits_{i\to \infty}  I_{{\epsilon_i}}[ u_i, c_i,D^-_{d/2,\epsilon_i}]> d\k(e_y).$
It follows from Remark \ref{rmk:liminfto0} and (\ref{eqn:liminf0}) that 
\begin{align*}
4d\k(e_y) = & \lim_{i\to \infty} I_{{\epsilon_i}}[ u_i, c_i,(-2d,2d)\times (-l,3l/4)] \\
\geq & \liminf_{i\to \infty}  I_{{\epsilon_i}}[ u_i, c_i,((-2d,-d/2)\cup(d/2,2d))\times (-l,3l/4)]  \\
& + \limsup_{i\to \infty}  I_{{\epsilon_i}}[ u_i, c_i,(-d/2,d/2)\times (-l,3l/4)] \\
  > & 3d\k(e_y) + d\k(e_y) = 4d\k(e_y),
\end{align*}
where in the second inequality we used Proposition \ref{FonsecaTartarProp} and horizontal translation.
This contradiction proves 
\begin{equation}\label{lim:A1}
\lim\limits_{i\to \infty}  I_{{\epsilon_i}}[ u_i, c_i,D^-_{d/2,\epsilon_i}] = d\k(e_y).
\end{equation}

Choosing $\theta$ sufficiently close to $1$, by Lemma \ref{lem:setIntersect} below, we find that $E_{i}\cap(\cap_j E_{j,\theta})\neq \emptyset$, and thus there is $y_i$ such that all previous bounds are simultaneously satisfied. It follows that $\bar u_i \to \bar u$ in $H^1(D_{d/2},\R^2)$ (unknown till now as we needed estimates for $a_i$ and $\phi_i$) and $\bar c_i \to \bar c$ in $L^2(D_{d/2}).$ Utilizing energy bounds (\ref{bdd:A2}), (\ref{bdd:A3}), (\ref{bdd:A4}), (\ref{bdd:A5}), convergence of $\eta_i$ (\ref{def:eta1}), and convergence of $A_1$ (\ref{lim:A1}), we find that $$ \lim_{i\to \infty}  I_{{\epsilon_i}}[ u_i, c_i,D_{d/2}] = d\k(e_y),$$ concluding the theorem.
\end{proof}

\begin{lem}\label{lem:setIntersect}
Suppose $E_i,$ $i=0,\ldots, k,$ are measurable subsets of $[0,1],$ and $\lambda \in (0,1).$ Then there is $\epsilon_0 = \epsilon_0 (\lambda, k)$ such that if $\mathcal{L}^1(E_0)>\lambda$ and $\mathcal{L}^1(E_i)>1-\epsilon$ for some $0<\epsilon<\epsilon_0$ for all $i=1,\ldots, k$, then 
\begin{equation}\label{eqn:setIntersect}
\bigcap_{i=0}^k E_i \neq \emptyset.
\end{equation} 
\end{lem}

\begin{proof}
Using subadditivity, we have
$$\mathcal{L}^1(\cap_{i>0} E_i) = 1-\mathcal{L}^1(\cup_{i>0}E_i^C) \geq 1-k\epsilon.$$ Take $\epsilon_0<\lambda/k.$ 
If (\ref{eqn:setIntersect}) does not hold, $$\mathcal{L}^1(\cap_{i\geq 0} E_i )= \mathcal{L}^1(E_0) + \mathcal{L}^1(\cap_{i>0} E_i)>\lambda + (1-\lambda)  = 1, $$
a contradiction.
\end{proof}

\begin{lem}\label{traceext} (see \cite{LeoniBook})
Given $d,l>0$ and $g\in H^{1/2}((-d,d)\times\{0\}),$ we may construct $v \in H^{1}((-d/2,d/2)\times (0,l))$ satisfying 
\begin{align*}
& v =  g \ \ \text{ on } (-d/2,d/2)\times\{0\} \nonumber\\
& v =  0 \ \ \text{  on some neighborhood of }\{(x,y): y = l \} \nonumber\\
& \|v\|_{H^1({(d/2,d/2)\times (0,l)})}^2 \leq  C \|g\|_{H^{1/2}((-d,d)\times\{0\})},
\end{align*}
for some constant $C>0$ independent of $g.$
\end{lem}

\begin{proof}
With an abuse of notation we treat $g$ as a function of $t\in (-d,d).$ Let $\eta:=\min\{d,l\}>0.$
Let $\phi \in C^{\infty}_c((-1,1))$ be a standard mollifier. For $(x,y)\in (-d/2,d/2)\times (0,\eta/2)$ we define $$\bar v(x,y) := \frac{1}{y}\int_{-d}^{d}\phi((x-t)/y)g(t)\ dt.$$ Since $\phi$ is even, $\int_{-d}^d \partial\phi((x-t)/y) \ dt = 0,$
so 
\begin{align*}
\frac{\partial \bar v}{\partial x}(x,y) = &  \frac{1}{(y)^2}\int_{-d}^{d}\partial\phi((x-t)/y)g(t)\ dt \\
= & \frac{1}{(y)^2}\int_{-d}^{d}\partial\phi((x-t)/y)[g(t)-g(x)]\ dt.
\end{align*} 
Consequently,
$$\Big| \frac{\partial \bar v}{\partial x}(x,y) \Big| \leq  \frac{C}{(y)^2}\int_{B(x,y)}|g(t)-g(x)| \ dt.$$

By H\"older's inequality and Fubini's Theorem
\begin{align*}
\int_{(-d/2,d/2)\times (0,\eta/2)} & \Big|\frac{\partial \bar v}{\partial x}(x,y)\Big|^2 \ d(x,y)  \\
\leq & C\int_{(-d/2,d/2)\times (0,d/2)}\frac{1}{(y)^4}\Big(\int_{B(x,y)}|g(t)-g(x)| \ dt\Big)^2 \ d(x,y)\\
\leq &  C\int_{(-d/2,d/2)\times (0,d/2)}\frac{1}{(y)^3}\int_{B(x,y)}|g(t)-g(x)|^2 \ dt \ d(x,y)\\
 \leq &  C\int_{(-d/2,d/2)}\int_{(-d,d)}|g(t)-g(x)|^2\Big(\int_{|t-x|}^\infty\frac{1}{(y)^3} \ dy \Big)\ dt \ dx \\
 = &  C\int_{(-d/2,d/2)}\int_{(-d,d)}\Big(\frac{|g(t)-g(x)|}{|t-x|}\Big)^2 \ dt \ dx \\
 \leq & C |g|_{H^{1/2}((-d,d)\times \{0\})}.
\end{align*}
Similarly, we compute
\begin{align*}
\frac{\partial \bar v}{\partial y}(x,y) = & \int_{-d}^{d}\frac{\partial}{\partial y}\Big( \frac{1}{y}\phi((x-t)/y)\Big)g(t)\ dt \\
= & \int_{-d}^{d}\frac{\partial}{\partial y}\Big( \frac{1}{y}\phi((x-t)/y)\Big)[g(t)-g(x)]\ dt, \\
\end{align*}
where in the last inequality we have used that for $(x,y)\in (-d/2,d/2)\times(0,\eta/2),$
\begin{equation}
 0  = \frac{\partial}{\partial y}(1) = \frac{\partial}{\partial y} \Big( \int_{-d}^{d} \frac{1}{y}\phi((x-t)/y) \ dt \Big)=  \int_{-d}^{d}\frac{\partial}{\partial y}\Big( \frac{1}{y}\phi((x-t)/y)\Big)\  dt. \nonumber
\end{equation}
We bound 
\begin{align*}
\Big|\frac{\partial}{\partial y}\Big( \frac{1}{y}\phi((x-t)/y)\Big)\Big| = & \Big|  -\frac{1}{(y)^2}\phi((x-t)/y) +\frac{(x-t)}{(y)^3}\partial\phi((x-t)/y)\Big| \\
\leq & \frac{C}{(y)^2} ,
\end{align*}
where we have used the fact that $|x-t|\leq y$ in the domain of integration. Thus we have
$$\Big|\frac{\partial \bar v}{\partial y}(x,y) \Big|\leq \frac{C}{(y)^2}\int_{B(x,y)}\|g(t)-g(x)\| \ dt, $$
and we may proceed as before. We conclude that
$$ \int_{(-d/2,d/2)\times (0,\eta/2)}  \Big\|\nabla \bar v(z)\Big\|^2 \ dz \leq C |g|_{H^{1/2}((-d,d)\times \{ 0 \})}.$$

Lastly, it remains to truncate the function, while preserving bounds. Let $\psi : \R \to [0,1]$ be a smooth function such that $\psi(t) = \chi_{(-\infty,1/2]}(t)$ for all $t \not \in [1/4,1].$ For any $\alpha>0,$ we define $v_\alpha(x,y) := \psi(y/\alpha)\bar v(x,y)$.
It is clear that  $$\int_{(-d/2,d/2)\times (0,\eta/2)}  \Big|\frac{\partial}{\partial x} v_\alpha(z)\Big|^2 \ dz \leq C |g|_{H^{1/2}((-d,d)\times\{0\})}$$ still holds.

We compute 
\begin{align*}
\int_{(-d/2,d/2)\times (0,\eta/2)}  \Big\|\frac{\partial}{\partial y} v_\alpha(z)\Big\|^2 \ dz \leq & C\int_{(-d/2,d/2)\times (0,\eta/2)}  \Big\|\frac{\partial}{\partial y} \bar v(z)\Big\|^2 \ dz \\
& + \frac{C}{\alpha^2}\int_{(-d/2,d/2)\times (0,\eta/2)}  \Big|\bar v(z)\Big|^2 \ dz.
\end{align*}
Using Fubini's/Tonelli's Theorem, it is straightforward to show that 
\begin{equation} \label{l2bdd}
\int_{(-d/2,d/2)\times (0,\eta/2)}  \Big|\bar v(z)\Big|^2 \ dz \leq C\|g\|_{L^2(-d,d)\times \{0\}}^2.
\end{equation}  Consequently, for any $\alpha>0$, we have $$\int_{(-d/2,d/2)\times (0,\eta/2)}  \Big\|\nabla v_\alpha(z)\Big\|^2 \ dz \leq  C_\alpha \|g\|_{H^{1/2}((-d,d)\times\{0\})}^2.$$ Choosing $\alpha$ sufficiently small based on the geometry of the domain, we conclude the lemma by setting $v = v_\alpha$. Note the desired $L^2$ bound follows from inequality (\ref{l2bdd}).
\end{proof}

\subsection*{Proof of Step II}
In this section, we use similar methods of proof as in the paper of Conti and Schweizer (Proposition 5.5 of \cite{ContiSchweizer-Linear}). We first prove a lemma relating energies to a geodesic distance similar to that of the Modica-Mortola functional.
In what follows, given a curve $\gamma,$ we interchangeably use $\gamma$ as the set and parameterization representing the curve.
\begin{lem}\label{energylines}
Let $g_\epsilon$ be defined as in (\ref{def:g_eps}). For any $\delta>0$ there is $h(\delta)>0$ such that if $\gamma$ is a $C^1$-curve with length at least $\epsilon,$ range in $(-d,d)\times (-l,l)$, and $\int_\gamma g_\epsilon \ d\mathcal{H}^1\leq h(\delta),$ then either $|c(x,y)-\mu_1|\leq \delta$ or $|c(x,y)-\mu_0|\leq \delta$ for all $(x,y)\in \gamma.$ 
\end{lem}

\begin{proof}

Consider the geodesic distance between points on the interval $I := [0,1]$ defined by 
\begin{equation}\label{def:geoDist}
d_I(s,s') := \inf\Big\{\int_0^1 \sqrt{f(\psi(t))}\|\nabla \psi(t)\|\ dt: \psi\in C^1(I, I), \psi(0)=s,\psi(1)=s' \Big\}. 
\end{equation}

Let $$h_0  :=\inf \{d_I(s,s'):s,s'\in I, |s-\mu_0|\leq \delta/2, |s'-\mu_0|\geq \delta\},$$ and similarly, $$h_1  :=\inf \{d_I(s,s'):s,s'\in I, |s-\mu_1|\leq \delta/2, |s'-\mu_1|\geq \delta\}.$$  Lastly, we define $$h_2  := \inf\{f(s):x\in I,|s-\mu_1|\geq \delta/2, |s-\mu_0|\geq \delta/2\}. $$ Let $h(\delta) :=\frac{1}{2}\min\{h_0,h_1,h_2\}.$

Assuming now that $\int_\gamma g_\epsilon \ d\mathcal{H}^1 <h(\delta)$ and $\mathcal{H}^1(\gamma)>\epsilon$, we have $$h_2 >  \int_\gamma g_\epsilon \ d\mathcal{H}^1 \geq \frac{1}{\epsilon} \inf\{f(c(x,y)):(x,y)\in \gamma\}\mathcal{H}^1(\gamma)\geq \inf\{f(c(x,y)):(x,y)\in \gamma\},$$ which implies there must be a point $(\bar x,\bar y)\in \gamma$ such that either $|c(\bar x,\bar y)-\mu_1|\leq \delta/2$ or $|c(\bar x,\bar y)-\mu_0|\leq \delta/2.$ Without loss of generality, assume that the latter holds.

By (\ref{def:g_eps}), we compute $$h_0>\int_\gamma g_\epsilon \ d\mathcal{H}^1 \geq \int_\gamma \sqrt{f(c)}\|\nabla c\|\ d\mathcal{H}^1  \geq \int_0^1 \sqrt{f(c\circ \bar \gamma)}\|\nabla (c\circ \bar \gamma)\| \ dt\geq d_I(c(x,y),c(\bar x,\bar y)),$$ where $(x,y)\in \gamma$ and $\bar \gamma$ is a curve contained in $\gamma$ connecting $(x,y)$ and $(\bar x,\bar y)$. By definition of $h_0,$ this implies $|c(x,y)-\mu_0|\leq \delta$ for all $(x,y)\in \gamma$ as desired.

\end{proof}

As in the proof of Proposition \ref{FonsecaTartarProp}, via a diagonalization argument, for any domain $ (-d,d)\times (-l,l)$, we may find sequences $\bar \epsilon_i\searrow 0$, $\bar u_i\to \bar u_{e_y}$ in $H^1((-d,d)\times (-l,l),\R^2)$, and $\bar c_i\to \bar c_{e_y}$ in $L^2((-d,d)\times (-l,l))$ such that 
\begin{equation}\label{lim:fixedSeqonBox}
\lim\limits_{i\to \infty} I_{{\bar \epsilon_i}}[\bar u_i,\bar c_i,(-d,d)\times (-l,l)] = 2d\k(e_y).
\end{equation} However with respect to gamma convergence, the sequence $\epsilon_i$ is given a priori. Hence the critical result is the following:

\begin{thm}\label{boxseq}
Assume (\ref{coercivity}), (\ref{detcond2}), and (\ref{fcoeffCond}) hold. Let $\epsilon_i\to 0,$ $l>0$, and $d>0$. There exist sequences $u_i\to \bar u_{e_y}$ and $c_i\to \bar c_{e_y}$ such that 
\begin{equation}\label{limsupineq}
\lim\limits_{i\to \infty} I_{{ \epsilon_i}}[ u_i, c_i,(-d/2,d/2)\times(-l,l)] = d\k(e_y).
\end{equation} 

\end{thm}

\begin{proof}
For notational convenience, we drop the subscript $e_y.$ Let $\bar \epsilon_i\searrow 0$, $\bar u_i\to \bar u$, and $\bar c_i\to \bar c$ be the sequences prior to the theorem statement for the domain $(-4d,4d)\times (-l,l)$. By Theorem \ref{boxenergy}, we find sequences (not relabeled) $\{\bar c_i\}\subset L^2((-d,d)\times (-l,l)) $ and $\{\bar u_i\}\subset H^1((-d,d)\times (-l,l),\R^2) $ such that on the upper and lower boundaries of $(-d,d)\times (-l,l),$ $\bar c_i =\bar c$ and $\bar u_i = \bar u + \chi_{y>0}w_i$, where $w_i$ is a skew affine function, with $$\lim\limits_{i\to \infty} I_{{\bar \epsilon_i}}[\bar u_i,\bar c_i,(-d,d)\times (-l,l)] = 2d\k(e_y).$$ Thus we extend $\bar c_i$ and $\bar u_i$ to $(-d,d)\times \R$ via constants or affine functions.

For each $i\in \N,$ we let $j(i)\in \N$ be the smallest number such that $j(i)>i$ and $\bar \epsilon_{j(i)} <\epsilon_i/i.$ We then rescale our sequences as follows:
$$\bar v_i(x,y) := \frac{\epsilon_i}{\bar \epsilon_{j(i)}}\bar u_i\Big(\frac{\bar \epsilon_{j(i)}}{\epsilon_i}(x,y)\Big), \ \ \bar b_i(x,y) := \bar c_i\Big(\frac{\bar \epsilon_{j(i)}}{\epsilon_i}(x,y)\Big).$$ Letting $\alpha_i :=  \frac{\epsilon_i}{\bar \epsilon_{j(i)}}$ and using a change of variables, we find $$ I_{{\epsilon_i}}[\bar v_i,\bar b_i,(-\alpha_i d,\alpha_i d)\times \R] = 2 \alpha_i d \k(e_y) +\alpha_i \eta_{j(i)}, $$
where $\eta_i := I_{{\bar \epsilon_i}}[\bar u_i,\bar c_i,(-d,d)\times (-l,l)] -2d \k(e_y).$ Thus 
\begin{align*}
\sum_{k =0 }^{\lfloor \alpha_i\rfloor - 1} I_{{\epsilon_i}}[\bar v_i,\bar b_i,(2k - \lfloor \alpha_i\rfloor)  d,(2(k+1) - \lfloor \alpha_i\rfloor) d)\times \R] = &  I_{{\epsilon_i}}[\bar v_i,\bar b_i,(-\lfloor \alpha_i\rfloor d,\lfloor \alpha_i\rfloor d)\times \R] \\
\leq & 2 \alpha_i d \k(e_y) +\alpha_i \eta_{j(i)}, 
\end{align*}
which implies there is some $k_0 \in  \{- \lfloor \alpha_i\rfloor,- \lfloor \alpha_i\rfloor + 2  , \ldots, \lfloor \alpha_i\rfloor  -2\}$ such that 

$$ I_{{\epsilon_i}}[\bar v_i,\bar b_i,(k_0d,(k_0+2)d)\times \R] \leq 2 \frac{\alpha_i}{\lfloor \alpha_i\rfloor} d \k(e_y) +\frac{\alpha_i}{\lfloor \alpha_i\rfloor} \eta_{j(i)}.$$
Translating the sequences, we assume $k_0 = -1.$ Taking the $\limsup$ of the previous inequality, we find 
\begin{equation}\label{extdomenergy}
\limsup\limits_{i\to \infty} I_{{\epsilon_i}}[\bar v_i,\bar b_i,(-d,d)\times \R] \leq 2 d \k(e_y),
\end{equation} as $\frac{\alpha_i}{\lfloor \alpha_i\rfloor} \to 1.$ Note further that associated to each sequence $\{\bar v_i,\bar b_i\}$ is some $L_i>0$ such that $\bar v_i$ is affine and $\bar b_i$ is constant in each of the connected regions specified by the inequality $|y|>L_i$. From this last fact, we are able to conclude that for each $i\in \N$, $I_{{\epsilon_i}}[\bar v_i,\bar b_i,(-d,d)\times \R]\geq Cd$ (see (\ref{transitionenergy})).

We now work to truncate the domain under consideration from $(-d,d)\times \R$ to $(-d,d)\times (-L,L)$ for some $L>0$ such that 
\begin{align}\label{desiredbound}
Cd & \leq\liminf\limits_{i\to \infty} I_{{\epsilon_i}}[u_i,c_i,(-d,d)\times (-L,L)] \\
& \leq \limsup\limits_{i\to \infty} I_{{\epsilon_i}}[u_i,c_i,(-d,d)\times (-L,L)] \leq 2 d \k(e_y) ,\nonumber
\end{align} where $u_i$ and $c_i$ are constructed from modifications of $\bar v_i$ and $\bar b_i$ and $u_i \to \bar u$ in $H^1((-d,d)\times (-L,L),\R^2)$, and $c_i \to \bar c$ in $L^2((-d,d)\times (-L,L))$.

In this direction, we let $\delta := (\frac{1}{2}-\mu_0)/2$ and define the functions 
\begin{equation}\label{def:f0}
f_0 (y) := \mathcal{L}^1(\{x\in (-d,d):|\bar b_i(x,y) - \mu_0|\leq \delta\}) 
\end{equation} and 
\begin{equation}\label{def:f1}
f_1 (y) := \mathcal{L}^1(\{x\in (-d,d):|\bar b_i(x,y) - \mu_1|\leq \delta\}). 
\end{equation}
For large $y,$ $f_0(y)= 0$ and $f_1(y)= 2d.$ An analogous situation holds for $y<<0.$ We utilize these functions to isolate an interval where (\ref{desiredbound}) will hold up to translation.

Note that the set of $y$ satisfying $f_0(y)+f_1(y)< 3d/2$ has Lebesgue measure less than $C_1\epsilon_i\leq C_1.$ To see this, note that if $f_0(y)+f_1(y)< 3d/2$, then \begin{equation}\label{condeqn}
\mathcal{L}^1(\{x\in (-d,d):|\bar b_i(x,y) - \mu_1|>\delta \text{ and }|\bar b_i(x,y) -\mu_0|>\delta \})>d/2.
\end{equation} This implies 
\begin{align*}
\frac{d}{2}\mathcal{L}^1(\{y: & \text{ inequality }(\ref{condeqn})\text{ holds}\}) \\
\leq & \int_{\R} \mathcal{L}^1(\{x\in (-d,d):|\bar b_i(x,y) - \mu_1|>\delta \text{ and }|\bar b_i(x,y) -\mu_0|>\delta \}) \ dy \\
\leq& C \int_{(-d,d)\times \R}f(\bar b_i) \ dz \leq  C_1\epsilon_i,
\end{align*}
where we have used that $f\geq 0$ with $f(c) =0$ if and only if $c=\mu_0$ or $c=\mu_1$.

We further note that the set on which both $f_0>0$ and $f_1>0$ is bounded in measure by a constant $C_2$. To see this, we use (\ref{def:g_eps}) to write 
\begin{equation}\label{eqn:c1ineq}
C_1\geq I_{{\epsilon_i}}[\bar v_i,\bar b_i,(-d,d)\times \R]  = \int_\R \int_{(-d,d)} g_{\epsilon_i}(x,y) \ dx \ dy.
\end{equation}

By Lemma \ref{energylines}, if $\int_{(-d,d)}g_{\epsilon_i}(x,y) \ dx \leq h(\delta),$ then either $f_0(y)$ or $f_1(y)$ is $0.$ Thus, we are concerned in bounding $$\mathcal{L}^1\left(\left\{y\in \R:\int_{(-d,d)}g_{\epsilon_i}(x,y) \ dx > h(\delta)\right\}\right).$$ But by Markov's Inequality and (\ref{eqn:c1ineq}), $$\mathcal{L}^1\left(\left\{y\in \R:\int_{(-d,d)}g_{\epsilon_i}(x,y) \ dx > h(\delta)\right\}\right)\leq C_1/h(\delta).$$ Thus $$\mathcal{L}^1(\{y: f_0(y)+f_1(y)<3d/2\}\cup\{y: f_0(y)>0, f_1(y)>0\})\leq C_1+C_1/h(\delta).$$ 
It follows that we may write $\R$ as the disjoint union of the three sets $ M$, $N$, and $ O,$ where
\begin{itemize}
\item $f_0 = 0$ and $f_1 >3d/2$ on $M$.
\item $f_0>3d/2$ and $f_1=0$ on $N$.
\item The remaining portion of $\R$ is $O$ with $\mathcal{L}^1(O) \leq C_1+C_1/h(\delta).$
\end{itemize}
Suppose $y_0$ and $y_1$ are such that $f_0(y_0)>3d/2$ and $f_1(y_1)>3d/2.$ Then by (\ref{def:f0}) and (\ref{def:f1}), the set$$E:=\{x\in (-d,d):|\bar b_i(x,y_0) - \mu_0|\leq \delta\}\cap\{x\in (-d,d):|\bar b_i(x,y_1) - \mu_1|\leq \delta\} $$ satisfies $$\mathcal{L}^1(E)>d.$$
Assuming without loss of generality $y_0<y_1$, we compute 
\begin{align}\label{transitionenergy}
\begin{split}
I_{{\epsilon_i}}[\bar v_i,\bar b_i,(-d,d)\times (y_0,y_1)]\geq& \int_E\int_{y_0}^{y_1}g_{\epsilon_i}\ dy\ dx \\
\geq& \inf\{d_I(c,c'): |c-\mu_0|\leq \delta,|c'-\mu_1|\leq \delta\}d=C_\delta d,
\end{split}
\end{align}
where $d_I$ is the geodesic distance from Lemma \ref{energylines} (see (\ref{def:geoDist})) and $C_\delta>0$.

If we refer to an interval $(y_0,y_1)$ as above as a \textbf{transition}, the energy bound (\ref{extdomenergy}) implies there are at most $J$ (independent of $i$) transitions.

Note that $ (-\infty,-L_i)\subset N\subset \ (-\infty,L_i]$ by (\ref{def:f0}) and (\ref{def:f1}) and the comment following these definitions. Hence we can define $$\bar y := \inf\{y: (y-\zeta,y)\cap N = \emptyset\}\geq -L_i > -\infty,$$ where $\zeta>2(C_1+C_1/h(\delta))$ (the constant makes sure at most half the interval is in $O$). For some $L>0,$ we consider the interval $  (\bar y -2L,\bar y -2\zeta),$ and divide the interval into segments of length $\zeta$ (assuming $2L$ is divisible by $\zeta$). 
Each interval intersects $N$ by definition of $\bar y.$ If an interval also intersects $M$, it contains a transition. Thus for $2L > (J+2)\zeta,$ there must be at least one such interval, $(\bar z,\bar z+\zeta),$ which does not intersect $M$, as the number of transitions must be less than $J$. 
Consequently, in this interval, for at least half the $y\in (\bar z,\bar z+\zeta),$ $f_1(y) = 0.$ We note this implies 
\begin{equation}\label{essinfest1}
\essinf_{x\in (-d,d)} |\bar b_i(x,y)-\mu_1|\geq \delta
\end{equation} for at least half the $y\in (\bar z,\bar z+\zeta).$ Similarly, we have 
\begin{equation}\label{essinfest2}
\essinf_{x\in (-d,d)} |\bar b_i(x,y)-\mu_0|\geq \delta
\end{equation} for at least half the $y\in (\bar y - \zeta,\bar y).$ 

We consequently define $$v_i(x,y) := \bar v_i(x,y-L+\bar y), \ \ b_i(x,y) := \bar b_i(x,y-L+\bar y),$$ for $(x,y)\in (-d,d)\times (-L,L)=:U_L.$
By construction, there must be at least one transition on the interval $(-L,L)$, and consequently, these sequences satisfy (\ref{desiredbound}). It remains to prove convergence.

We define $$\eta_i := \inf\{\|v_i -u_0\|_{H^1(U_L)}+\|b_i-c_0\|_{L^2(U_L)}: (u_0,c_0)\in \mathcal{G}\}, $$ where 
\begin{align*}
\mathcal{G} := \{& (u_0,c_0) \in H^{1}(U_L,\R^2)\times L^2(U_L): \ u_0(x,y) = \bar u (x,y-a)+S(x,y)^T+r,  \\
& c_0(x,y) = \bar c(x,y-a), \text{for all }(x,y)\in U_L\text{, and} \\
& a\in (-L+\zeta/2,L-\zeta/2), S\in \R^{2\times 2}_{\text{skew}},r\in \R^2\} .
\end{align*}
We claim $\eta_i\to 0.$ If not, there is a subsequence $\{\eta_{i_k}\}$ bounded away from $0.$ Considering the compactness Theorem \ref{thm:compactness}, we have that $v_{i_k} \to v$ in $\in H^1(U_{L}, \R^2)$ and  $b_{i_k} \to b\in BV(U_L,\{\mu_0,\mu_1\})$ in $L^2(U_L)$, with $e(v) = be_0.$ Without loss of generality, we may assume that $b_{i_k}\to b$ pointwise $a.e.$, and consequently, $b$ satisfies (\ref{essinfest1}) and (\ref{essinfest2}). By Theorem \ref{struct}, we have that $v$ only has horizontal or vertical interfaces. By the essential infimum estimates (\ref{essinfest1}) and (\ref{essinfest2}), there are no vertical interfaces. By the energy bounds (\ref{desiredbound}), $v$ can have at most one horizontal interface transition. Once again by the essential infimum estimates, $b(x,y) = \mu_1$ for $y>L-\zeta/2$ and $b(x,y) = \mu_0$ for $y<-L+\zeta/2$, else we contradict $L^2$ convergence results. We conclude that $b = \bar c(x,\cdot-a)$ for some $a\in (-L+\zeta/2,L-\zeta/2).$ It follows $(v,b)\in \mathcal{G},$ which then contradicts the assumption $\liminf_k\eta_{i_k}>0.$ 

We conclude that $\eta_i\to 0.$ Translating functions and shifting by affine functions with skew gradient, we find $u_i:(-d,d) \times (-\zeta/2,\zeta/2)\to \R^2$ and $c_i:(-d,d)\times (-\zeta/2,\zeta/2)\to [0,1]$ satisfying the conclusion of the theorem with $l = \zeta/2$. Applying Theorem \ref{boxenergy}, we obtain the theorem's conclusion for $l = \zeta/2$ where $u_i$ and $c_i$ are affine or constant (respectively) on the upper and lower boundaries. Extending these functions to be affine or constant, the theorem's conclusion holds on $(-d/2,d/2)\times \R,$ which may then be truncated to the desired domain $(-d/2,d/2)\times(-l,l).$ 

\end{proof}
\subsection*{Proof of Step III}

\begin{proof}[Proof of Theorem \ref{thm:charInterface}]
Apply Theorem \ref{boxseq} to the domain $(-2d,2d)\times (-l,l)$. Subsequently, apply Theorem \ref{boxenergy} to conclude the result.
\end{proof}

\section{Limsup bound}\label{seclimsup}

We outline our plan to prove the $\limsup$ bound on a strictly star-shaped Lipschitz domain $\Omega$. In essence, we wish to put boxes around the interfaces, and interpolate between the sides of the boxes parallel to the interface by low energy sequences while maintaining regularity of the functions. More explicitly:

\begin{itemize}
\item Given $u$ and $c$ for which $I_0[u,c,\Omega]$ is finite, we rescale the functions utilizing the fact that the domain is strictly star-shaped. This reduces the problem to the case of finitely many interfaces.
\item Suppose without loss of generality that some interface has normal $e_y.$ Around this interface, we intersect the domain with a box of small width in the normal direction. For a given sequence $\epsilon_i$, in each box, we use the characterization of the interfacial energy to construct a sequence of functions such that $I_{\epsilon_i}[u_i,c_i,(-d,d)\times(-l,l)]\to 2d\k(e_y),$ and  both $ u_i $ is affine and $c_i$ is constant on the boundaries of the box parallel to the interface.
\item We use the previous step to construct a low energy sequence which is equal to $u$ plus a ``small" skew affine function outside of the boxes and in the box is equal to the low energy sequence with affine boundary conditions.
\end{itemize}

\begin{thm} \label{limsupthm} (see also Proposition 5.1 of \cite{ContiSchweizer-Linear}) Assume (\ref{coercivity}), (\ref{detcond2}), and (\ref{fcoeffCond}) hold. Suppose $\epsilon_i\to 0$ and that $\Omega$ is an open, strictly star-shaped domain with Lipschitz continuous boundary. For $u\in H^1(\Omega;\R^2)$ and $c\in BV(\Omega;\{\mu_0,\mu_1\})$ with $I_0[u,c,\Omega]<\infty$, there are sequences $u_i\to_{H^1} u$ and $c_i \to_{L^2} c$ such that $$\limsup_{i \to \infty} I_{{ \epsilon_i}}[ u_i, c_i,\Omega] \leq I_0[u,c,\Omega].$$
\end{thm}

\begin{proof}
Assume without loss of generality that $\Omega$ is star-shaped about $0.$
Given $\theta\in(0,1),$ we rescale $u$ and $c$ to define $$u_\theta(x,y) := \frac{1}{\theta}u(\theta (x,y)), \ \ c_\theta(x,y) := c(\theta (x,y)), \ \  \text{ for }  (x,y)\in \Omega.$$
We prove
\begin{equation}\label{covtheta}
\limsup_{i \to \infty} I_{{ \epsilon_i}}[ u_i, c_i,\Omega]\leq \frac{1}{\theta}I_0[u,c,\Omega],
\end{equation}
for sequences $u_i\to u_{\theta}$ in $H^1(\Omega,\R^2)$ and $c_i \to c_{\theta}$ in $L^2(\Omega).$

Supposing we prove this for $u_\theta$ and $c_\theta,$ we may consider a sequence $\theta_k\to 1$ and find subsequences $\{u_{i,k}\}_{i}$ and $\{c_{i,k}\}_{i}$ such that satisfying inequality (\ref{covtheta}). Taking the $\limsup$ with respect to $k$ of the above inequality, we may apply a diagonalization argument to conclude the theorem.

Thus it remains to prove (\ref{covtheta}) for fixed $\theta.$ By Theorem \ref{struct}, $J_c = \cup_j S_j,$ where each $S_i$ is a connected segment parallel to one of the axes. Thus $J_{c_\theta} = \cup_j(\Omega \cap \frac{1}{\theta}S_j) =:\cup_j S_{j,\theta}.$ We note that $\text{dist}(S_{j,\theta},S_{m,\theta})>0$ for $j\neq m$ as $\bar S_j$ and $\bar S_m$ can only intersect at endpoints, and thus the strict star-shapedness implies, $\bar S_{j,\theta}\cap \bar S_{m,\theta} = \emptyset.$ 

Furthermore, we have that $S_{j,\theta} = \emptyset$ for all but finitely many of the $j$. Supposing not, we may find a sequence $z_k $ such that $z_k\in S_{j_k}\cap \theta \Omega$ for a strictly increasing sequence $\{j_k\}_k.$ As $\mathcal{H}^1(J_c)<\infty,$ $\mathcal{H}^1(S_{j_k})\to 0.$ It follows that up to a subsequence $z_k\to z_0\in \partial \Omega.$ But by choice of $z_k,$ we have $z_0\in \theta \bar \Omega.$ This is a contradiction as $\partial \Omega \cap \theta \bar \Omega = \emptyset$ by strict star-shapedness. 

From here on we only consider $j$ for which $S_{j,\theta}$ is nonempty. Consider a horizontal segment, $S_j = (x_j^-,x_j^+)\times \{y_j\}.$ By strict star-shapedness $\frac{1}{\theta}(x_j^{\pm}\times\{y_j\})\not \in \bar \Omega.$ Thus we may find $\sigma>0$ such that $\{\frac{1}{\theta}x_j^{\pm}\}\times (\frac{1}{\theta}y_j-\sigma,\frac{1}{\theta}y_j+\sigma)\cap \bar \Omega = \emptyset.$ We let $R_{j} := (\frac{1}{\theta}x_j^{-},\frac{1}{\theta}x_j^{+})\times (\frac{1}{\theta}y_j-\sigma,\frac{1}{\theta}y_j+\sigma).$ Similarly, we define $R_{j}$ for vertical interfaces. For $\sigma$ sufficiently small, the sets $R_j\cap \Omega$ are disjoint.

Associated to each $R_{j}$ is unit normal $\nu_j$ and, as given by Theorem \ref{thm:charInterface}, there is a sequence with $\{u_i^j,c_i^j\}_{i}$ with $u_i^j = u + \chi_{\nu_j \cdot (x,y)>0}(R_{\phi_i^j}(x,y)^T+a_i)$ and $c_i^j = c$ on the boundaries of the box parallel to the interface and energy bounds as given by (\ref{limsupineq}). We now seek to define sequences $u_i$ and $c_i$.

We divide $\Omega \setminus (\cup_j S_{j,\theta})$ into connected components $\{\Omega_k\}.$ We induce a partially ordered system ($\prec$) on $\{\Omega_k\}$ to make it into a downward directed set. Up to reordering, let $\Omega_1$ be a connected component with boundary only touching one interface.  $\Omega_1$ is defined to be the minimal element in the POS ($\prec$). By star-shapedness, between every point of $\Omega_1$ and $\Omega_k$, there is a unique minimal sequence of connected components, $\{\Omega_{k_i}\}_{i=1}^n$, $k_1 = 1$ and $k_n = k$, through which a continuous path in $\Omega$ must travel to move between the points. We say $\Omega_{k_i}\prec \Omega_{k_{i+1}}$. Looking at all paths induces the desired POS ($\prec$). Note, we have that each $\Omega_k$ has a unique element $\Omega_{k'}$ which is the greatest element less than it. Letting $S_{j,\theta}$ be the interface separating the domains $\Omega_k$ and $\Omega_{k'}.$ We define $\phi^k_i : = \phi^j_i$ and likewise for $a_i^j$. Without loss of generality, we have that $\nu_j$ points from $\Omega_{k'}$ towards $\Omega_k$. Note we also treat ($\prec$) as a partial order on $\{k\}.$ With this, we define
$$u_i(x,y)  := \begin{cases} 
      u_i^j(x,y) +\sum_{n\prec k}(R_{\phi_i^n}(x,y)^T+a_i^n)& (x,y)\in R_j \cap \Omega, \Omega_k\cap R_j \neq \emptyset, \Omega_{k'}\cap R_j \neq \emptyset  \\
      u(x,y)+ \sum_{n\preceq k}(R_{\phi_i^n}(x,y)^T+a_i^n) & \text{ not in the previous case, and } (x,y)\in \Omega_k,
   \end{cases}   
$$
$$
   c_i  := \begin{cases} 
      c_i^j(x,y) & (x,y)\in R_j \cap \Omega, \\
      c & otherwise. 
   \end{cases}
 $$
It follows that 
\begin{align*}
\limsup_{i \to \infty} I_{{ \epsilon_i}}[ u_i, c_i,\Omega] \leq& \sum_j \limsup_{i\to \infty} I_{{ \epsilon_i}}[ u_i^j, c_i^j,R_j] \\
\leq& \frac{1}{\theta} \sum_j \k(\nu_{i})\mathcal{H}^1(S_i) \leq  \frac{1}{\theta}I_0[u,c,\Omega],
\end{align*}
proving the desired inequality (\ref{covtheta}). Convergence of the subsequences to $u_\theta$ and $c_\theta$ follows from convergence on the boxes and decay of $\phi_i^j$ and $a_i^j$ to $0$ (see Theorem \ref{boxenergy}).

\end{proof}

\section{Mass Constraint}\label{secmassconstraint}

We now treat the case of $\Gamma$-convergence under the restriction of a mass constraint. Recall that we let $\{m_\epsilon \}_{\epsilon>0}\subset [0,1]$ be a net converging to $m_0\in [\mu_0,\mu_1]$ as $\epsilon\to 0$, and we wish to consider $\Gamma$-convergence restricting $c_\epsilon$ to satisfy $\intbar_\Omega c_\epsilon = m_\epsilon.$ Obviously, the $\liminf$ bound still holds, and thus for given $\epsilon_i\to 0,$ it remains to show that we may construct a sequence obtaining the limit. We write $m_i$ for $m_{\epsilon_i}.$ We break this into cases depending on whether $m_0 =\mu_0,$ $m_0 =\mu_1,$ or $m_0\in (\mu_0,\mu_1)$. In each case, we need to find some way to fluctuate the mass of the functions $c_\epsilon.$ To do this, we will emulate the proof of $\limsup$ bound for the Modica-Mortola functional to construct a low energy perturbation of $c_\epsilon$ as previously constructed (see \cite{Leoni2013-SummerSchool}, \cite{ModicaMortola}).

\begin{proof}[Proof of Theorem \ref{maintheoremconstraint}]
Consider $(u,c)$ such that $I_0[u,c,\Omega]<\infty$ and $\intbar_\Omega c = m_0.$ We construct minimizing sequences for different cases.

\noindent \textbf{Case 1, $\boldsymbol{m_0 = \mu_0}$ or $\boldsymbol{m_0 = \mu_1}$:} Without loss of generality, we treat the case that $m_0 = \mu_0.$ Note that in this case, the function $c = \mu_0$ and $e(u) = \mu_0 e_0$. Thus if $m_i = \mu_0,$ we may simply choose $c_i = c.$ Consequently, in the following construction, we assume that $m_{i}\neq \mu_0$ for all $i.$

We consider the energy functionals given by $$\bar I_{\epsilon_i}[c',\Omega] :=I_{\epsilon_i}[u,c',\Omega] = \int_{\Omega}\frac{1}{\epsilon}\Big( f(c') + \|(c'-\mu_0)e_0\|^2 \Big)+\epsilon\|\nabla c'\|^2 \ dz.$$ We condense notation by defining $W(s):=f(s) + \|(s-\mu_0)e_0\|^2.$ Note this is a single-well potential. 

\begin{adjustwidth}{1cm}{}

\noindent \textbf{Subcase 1, $\boldsymbol{\mu_0<m_i<\mu_1}$:} We define the sequence $\{E_\eta\}_{\epsilon_{z_0}>\eta>0}$ by $E_\eta:= B(z_0,\eta)^C$, for any fixed $z_0\in \Omega$ such that $B(z_0,2\epsilon_{z_0})\subset \Omega$. Define $E_i : = E_{\eta_i} = B(z_0,\eta_i)^C$, where $\eta_i>0$ is such that 
\begin{equation}\label{def:etaSeq}
\mu_0\mathcal{L}^2(E_{\eta_i}\cap \Omega)+\mu_1\mathcal{L}^2(E_{\eta_i}^C\cap \Omega) = m_i \mathcal{L}^2(\Omega).
\end{equation} This assumes that $m_i$ is sufficiently close to $\mu_0$ (as given by some relation to $\epsilon_{z_0}$), which we do.

Define $$\phi_i (s):= \int_{\mu_0}^{s}\frac{\epsilon_i}{\sqrt{\epsilon_i +W(r) }} \ dr.$$ Then, $|\phi_i(\mu_1)|\leq \epsilon_i^{1/2}.$
We note that $\phi_i$ is strictly increasing with differentiable inverse $\phi_{i}^{-1}:[0,\phi_i(\mu_1)]\to [\mu_0,\mu_1]$ satisfying $$\frac{d}{dt}\phi_i^{-1}(t) = \frac{\sqrt{\epsilon_i+W(\phi_i^{-1}(t))}}{\epsilon_i},$$ by the inverse function theorem. Extend $\phi_i^{-1}$ by constants at the boundary of $[0,\phi_i(\mu_1)]$.
We define 
\begin{equation}\label{def:g0MM}
g_0(t):=
\begin{cases} 
\mu_0 ,& t\leq 0, \\
\mu_1 ,& t> 0, 
\end{cases}
\end{equation}
and $$v_s(z) : = \phi_i^{-1}(d_{E_i}(z)+s),$$ where $$d_{E_i}(z):= \begin{cases} -d(z,\partial E_i)  & \text{if } z\in E_i, \\ \phantom{+} d(z,\partial E_i) & \text{otherwise,} \end{cases}$$ is the signed distance function of $E_i$ (negative in $E_i$).

We now wish to choose $s$ such that the $\intbar_{\Omega} v_{s_i} \ dz = m_i.$ To do this, we apply the Mean Value theorem to the function $s\mapsto \intbar_{\Omega} v_s \ dz$. We compute
\begin{align*}
\intbar_{\Omega} \phi_i^{-1}(d_{E_i}(z)) \ dz \leq & \intbar_{\Omega} g_0(d_{E_i}(z)) \ dz = m_i, \\
\intbar_{\Omega} \phi_i^{-1}(d_{E_i}(z)+\phi_i(\mu_1)) \ dz \geq &\intbar_{\Omega} g_0(d_{E_i}(z)) \ dz = m_i.&
\end{align*}
Thus, for some $s_i\in [0,\phi_i(\mu_1)],$ we have $\intbar_\Omega v_{s_i} \ dz = m_i.$ Define $c_i:=v_{s_i}.$ We now wish to perform a precise estimate on $c_i.$ Since $d_{E_i}$ is Lipschitz continuous and $\|\nabla d_{E_i} (z)\| = 1$ for a.e. $z\in \R^2\sm \partial E_i$, (see \cite{EvansGariepy}, \cite{Leoni2013-SummerSchool}, \cite{Bellettini-LectureMeanCurv}) we can apply the coarea formula (see \cite{EvansGariepy},\cite{LeoniBook}) to obtain
\begin{align*}
\bar I_{\epsilon_i}[c_i,\Omega] = & \int_\Omega \frac{1}{\epsilon_i}W(\phi_i^{-1}(d_{E_i}(z)+s_i))+\epsilon_i\|\nabla (\phi_i^{-1}(d_{E_i}(z)+s_i))\|^2 \\
= &\int_{-s_i}^{\eta_i} \Big(\frac{1}{\epsilon_i}W(\phi_i^{-1}(r+s_i))+\epsilon_i|(\phi_i^{-1})'(r+s_i)|^2 \Big)  \mathcal{H}^1(\{z\in \Omega: d_{E_i}(z) = r\}) \ dr \\
\leq & \sup_{-s_i < t<\eta_i} \mathcal{H}^1(\{z\in \Omega: d_{E_i}(z) = t\}) \int_{0}^{\eta_i+s_i} \frac{1}{\epsilon_i}W(\phi_i^{-1}(r))+\epsilon_i|(\phi_i^{-1})'(r)|^2 \  dr \\
\leq & \sup_{-s_i < t<\eta_i} \mathcal{H}^1(\{z \in \Omega: d_{E_i}(z) = t\}) \int_{0}^{\eta_i+s_i} \frac{\epsilon_i+W(\phi_i^{-1}(r))}{\epsilon_i}+\epsilon_i|(\phi_i^{-1})'(r)|^2 \  dr \\
= &\sup_{-s_i < t<\eta_i} \mathcal{H}^1(\{z\in \Omega: d_{E_i}(z) = t\}) \int_{0}^{\eta_i+s_i} 2\sqrt{\epsilon_i + W(\phi_i^{-1}(r))}|(\phi_i^{-1})'(r)| \  dr \\
\leq  & \sup_{-s_i < t<\eta_i} \mathcal{H}^1(\{z\in \R^2: d_{E_i}(z) = t\}) \int_{0}^{1} 2\sqrt{\epsilon_i + W(s)}\ ds \\
\leq & C(\epsilon_i^{1/2}+\eta_i)\int_{0}^{1} 2\sqrt{\epsilon_i + W(s)}\ ds \to  0
\end{align*}
as $i\to \infty.$
We now check convergence in $L^2(\Omega)$ by the same means:
\begin{align*}
\int_\Omega |c_i-\mu_0|^2 =& \int_\Omega |\phi_i^{-1}(d_{E_i}(z)+s_i) - \mu_0|^2 \\
=&  \int_{-s_i}^{\eta_i} |\phi_i^{-1}(r+s_i)-\mu_0|^2 \ \mathcal{H}^1(\{z\in \Omega: d_{E_i}(z) = r\}) \ dr \\
\leq & (|s_i|+|\eta_i|)\sup_{-s_i < t<\eta_i} \mathcal{H}^1(\{z\in \R^2: d_{E_i}(z) = t\}) \to  0
\end{align*}
With this, we have proven $\Gamma$-convergence.

\noindent \textbf{Subcase 2, $\boldsymbol{m_i<\mu_0}$:} The proof is predominantly the same as the previous subcase. We comment on the changes. To define $\eta_i,$ consider $\mu_0\mathcal{L}^2(E_\eta\cap \Omega) = m_i$ in place of (\ref{def:etaSeq}). We use $0$ in place of $\mu_1$ in the definition of (\ref{def:g0MM}). 

\end{adjustwidth}
\noindent \textbf{Case 2, $\boldsymbol{m_0 \in (\mu_0,\mu_1)}$:} In this case, we know that $J_c\neq \emptyset,$ and further, there must be a point $z_0 \in \Omega$ such that $B(z_0,2\epsilon_{z_0})\subset \Omega$ and $B(z_0,2\epsilon_{z_0})\cap J_{c} = \emptyset.$ Thus by the construction in Theorem \ref{limsupthm}, we can find a low energy sequence $\{(u_i,c_i)\}_i$ converging to $(u,c)$ such that $c_i|_{B(z_0,\epsilon_{z_0})}e_0 = e(u)|_{B(z_0,\epsilon_{z_0})} = \mu_0 e_0$ for all $i$. Likewise, we can find $z_1 \in \partial \Omega$ such that $c_i|_{B(z_1,\epsilon_{z_1})}e_0 = e(u)|_{B(z_1,\epsilon_{z_1})} = \mu_1 e_0$ with $B(z_1,2\epsilon_{z_1})\subset \Omega$ and $B(z_1,2\epsilon_{z_1})\cap J_{c} = \emptyset.$

We note that $m_i\to m_0$, and $\intbar_\Omega c_i \ dz \to m_0.$ Supposing $\intbar_\Omega c_i \ dz <m_i,$ we perform the same procedure from the preceding section on $B(z_0,\epsilon_{z_0})$ to construct $c_{\phi,i}:B(z_0,\epsilon_{z_0})\cap \Omega \to [0,1]$ (utilizing $E_\eta = B(z_0,\eta)^C$) with mass 
$$\intbar_{B(z_0,\epsilon_{z_0}) } c_{\phi,i} \ dz  = \frac{m_i\mathcal{L}^2(\Omega)- \int_\Omega c_i \ dz}{\mathcal{L}^2(B(z_0,\epsilon_{z_0}))}+\mu_0,$$ 
(which makes sense for sufficiently large $i$) and 

$$\lim_i I_{\epsilon_i}[c_{\phi,i},u_i,B(z_0,\epsilon_{z_0})] = 0.$$

 We define
$$ \bar c_i (z)  := \begin{cases}
c_i & \text{if } z\in \Omega \setminus B(z_0,\epsilon_{z_0}), \\
c_{\phi,i} & \text{if } z\in \Omega \cap B(z_0,\epsilon_{z_0}),
\end{cases}$$
which satisfies $\bar c_i\to c$ in $L^2(\Omega)$ and is directly shown to satisfy $\intbar_\Omega \bar c_i \ dz = m_i$.

We note by Theorem \ref{thm:liminf} the sequence $(u_i,c_i)$ is of minimal energy on every Lipschitz subset of $\Omega$, and it follows $I_{\epsilon_i}[u_i,c_i, B(z_0,\epsilon_{z_0})\cap \Omega]\to 0$. Thus, 
$$\lim_{i\to \infty} I_{\epsilon_i}[u_i,c_i, \Omega] = \lim_{i\to \infty} I_{\epsilon_i}[u_i,\bar c_i, \Omega].$$

Similarly, if $\intbar_\Omega c_i \ dz >m_i,$ we would perform the analogous calculation about $z_1$ to decrease the mass of $c_i.$ Consequently, we have shown the desired $\Gamma$-convergence result.

\end{proof}

\section*{Acknowledgments} This paper is part of the author's Ph.D. thesis at Carnegie Mellon University under the direction of Irene Fonseca and Giovanni Leoni. The author is deeply indebted to these two for their many hours spent watching the author scribble at a board and for expert guidance on many mathematical topics. Furthermore, the author is thankful for their many suggestions as to the organization of the paper and spotting a plethora of typos, which greatly improved the paper. 
The author was partially supported by National Science Foundation Grants DMS 1906238 and DMS 1714098.

\bibliographystyle{amsplain}
\bibliography{kstinsonGammaBattery}


\end{document}